\documentclass[11pt]{amsart}

\usepackage{amsmath, amssymb, amsbsy, amscd, mathtools}
\usepackage{times, mathrsfs, euscript}
\usepackage{epsf, overpic, psfrag, stmaryrd}
\usepackage{caption, subcaption}
\usepackage[all]{xy}
\usepackage{hyperref}
\usepackage[dvipsnames]{xcolor}
\usepackage{enumitem, bbm}
\usepackage{tikz, tikz-cd}
\usetikzlibrary{decorations.markings}
\tikzcdset{scale cd/.style={
    every label/.append style={scale=#1},
    cells={nodes={scale=#1}}}}

\usepackage{array}
\usepackage{todonotes}
\usepackage{tikz}

\usepackage{bm}
\newtheorem{theorem}{Theorem}[section]

\newtheorem{claim}[theorem]{Claim}

\newtheorem{lemma}[theorem]{Lemma}
\newtheorem{proposition}[theorem]{Proposition}

\theoremstyle{definition}
\newtheorem{definition}[theorem]{Definition}

\newtheorem{remark}[theorem]{Remark}
\newtheorem{assumption}[theorem]{Assumption}

\numberwithin{equation}{subsection}

\DeclareMathAlphabet{\mathpgoth}{OT1}{pgoth}{m}{n}

\DeclareMathAlphabet{\mathpzc}{OT1}{pzc}{m}{it}

\newcommand{\W}{\mathbb{W}}

\newcommand{\be}{\begin{enumerate}}
\newcommand{\ee}{\end{enumerate}}
\newcommand{\op}{\operatorname}

\newcommand{\hh}{\mathbb H}

\newcommand{\del}{\partial} 

\newcommand{\dbar}{\bar\del}
\DeclareMathOperator{\ind}{Ind}

\newcommand{\iT}[1]{#1}

\newcommand{\obstructionSection}{\mathfrak{s}}

\newcommand{\chartData}[1]{\mathcal{C}_{\iT{I}}}

\input pdf-trans
\newbox\qbox
\def\usecolor#1{\csname\string\color@#1\endcsname\space}
\newcommand\bordercolor[1]{\colsplit{1}{#1}}
\newcommand\fillcolor[1]{\colsplit{0}{#1}}
\newcommand\outline[1]{\leavevmode%
  \def\maltext{\mydelim #1\mydelim}%
  \setbox\qbox=\hbox{\maltext}%
  \boxgs{Q q 2 Tr \bbthickness\space w \fillcol\space \bordercol\space}{}%
  \copy\qbox%
}

\newcommand\colsplit[2]{\colorlet{tmpcolor}{#2}\edef\tmp{\usecolor{tmpcolor}}%
  \def\tmpB{}\expandafter\colsplithelp\tmp\relax%
  \ifnum0=#1\relax\edef\fillcol{\tmpB}\else\edef\bordercol{\tmpC}\fi}
\def\colsplithelp#1#2 #3\relax{%
  \edef\tmpB{\tmpB#1#2 }%
  \ifnum `#1>`9\relax\def\tmpC{#3}\else\colsplithelp#3\relax\fi
}
\bordercolor{black}
\fillcolor{white}
\newcommand\bbthickness{.5}

\newcommand{\crit}{\op{Crit}}
\newcommand{\cP}{\mathsf} 

\newcommand{\grad}{\op{grad}}
\newcommand{\dd}{\mathrm{d}}

\newcommand{\C}{\mathbb{C}}
\newcommand{\R}{\mathbb{R}}

\newcommand{\iu}{\sqrt{-1}}

\newcommand{\M}{\mathcal{M}}

\newcommand{\tup}[1]{\bm{#1}}

\newcommand{\p}{\partial}

\newcommand{\td}{}

\newcommand{\lm}{w}
\newcommand{\tdlm}{{\td{\lm}}}
\newcommand{\lv}{\zeta}
\newcommand{\tdlv}{\td{\lv}}
\newcommand{\constantVector}{\td{\varpi}}

\newcommand{\app}{u_{\op{app}}}

\newcommand{\parameterSpace}{\mathcal{K}}

\newcommand{\norm}[1]{\left\lVert#1\right\rVert}
\newcommand{\eps}{\varepsilon}

\makeatletter
\newcommand{\labitem}[2]{%
  \def\@itemlabel{#1}
  \item
  \def\@currentlabel{#1}\label{#2}}
\makeatother

\begin{document}

\title[From Morse Trees to $J$-Holomorphic Discs -- Rigid Y-Graphs]{From Morse Trees to Pseudo-Holomorphic Discs -- Y-Graphs}

\author{Erkao Bao}
\address{School of Mathematics, University of Minnesota, Minneapolis, MN 55455}
\email{bao@math.umn.edu}
\urladdr{https://erkaobao.github.io/math/}
\thanks{Erkao Bao is supported by NSF Grants DMS-2404529.}

\author{Ke Zhu}
\address{Department of Mathematics and Statistics, Minnesota State University Mankato, Mankato, MN 56001}
\email{ke.zhu@mnsu.edu}
\urladdr{https://faculty.mnsu.edu/kezhu/}


\begin{abstract}
The correspondence between Morse flow trees and $J$-holomorphic discs was established by Fukaya--Oh \cite{fukaya1997zeroloop} and Ekholm \cite{ekholm2007morse}. We revisit this correspondence and present an alternative approach, designed to generalize naturally to the equivariant setting and to certain Morse graph configurations. The main ingredient is a gluing construction that produces $J$-holomorphic discs from Morse flow trees. A well-known difficulty is that this gluing is of Morse--Bott type; in other words, under an appropriate Fredholm framework, pieces to be glued together are obstructed. We resolve this via the obstruction bundle gluing technique of Hutchings--Taubes \cite{hutchings2009gluing}. Given a rigid, transversely cut-out Y-shaped Morse flow tree, we show that for every sufficiently small $\epsilon > 0$ there exists at least one corresponding $J$-holomorphic discs in the cotangent bundle, with boundaries inside corresponding Lagrangian submanifolds of height $\epsilon$. This is the first paper in a series; subsequent work will extend the result to all ribbon trees and to moduli spaces of all dimensions and establish the injectivity and surjectivity of the correspondence.
\end{abstract}

\maketitle

\setcounter{tocdepth}{1}
\tableofcontents

\section{Introduction and Main Results}

\subsection{Background}
The relationship between Morse theory and $J$-holomorphic curves has its roots in Floer's foundational work \cite{floer1989witten}, where holomorphic strips in the cotangent bundle $T^*M$ are shown to correspond bijectively to gradient trajectories of a Morse function on $M$. This correspondence was substantially generalized by Fukaya--Oh \cite{fukaya1997zeroloop}, who showed that moduli spaces of $J$-holomorphic polygons in $T^*M$ with boundary on conormal Lagrangians are diffeomorphic to moduli spaces of gradient flow trees on $M$, establishing a Morse-theoretic model for the $A_\infty$-structure of the Fukaya category in cotangent bundles. Ekholm \cite{ekholm2007morse} further extended the correspondence to the contact setting, showing that rigid pseudo-holomorphic discs in $T^*M$ correspond bijectively to rigid Morse flow trees of a Legendrian submanifold, giving a combinatorial algorithm for computing Legendrian contact homology. A parallel structure arises in spectral networks \cite{gmn2013spectral}, which encode BPS degeneracies and wall-crossing data through a combinatorial graph structure on a Riemann surface. The connection between spectral networks and holomorphic curve theory has been pursued more recently by Honda--Tian--Yuan \cite{honda2026spectral}. There are growing needs to generalize the above correspondences to the equivariant setting and to more general graph configurations, some of which have obstructions, and this motivates our alternative obstruction bundle gluing approach.

\subsection{Setup}

Let $M$ be a closed $n$-dimensional manifold. Fix a Riemannian metric $g$ on $M$ and a tuple of smooth functions $\tup{f} = (f_1, f_2, f_3)$ on $M$. We say $\tup{f}$ is \emph{Morse} if $f_a - f_b$ is a Morse function for all distinct indices $a, b$. A \emph{tuple of critical points} $\tup{\cP{p}} = (\cP{p}_1, \cP{p}_2, \cP{p}_3)$ consists of points
\[
\cP{p}_1 \in \crit(f_1 - f_3), \qquad
\cP{p}_2 \in \crit(f_2 - f_1), \qquad
\cP{p}_3 \in \crit(f_3 - f_2).
\]

\subsubsection*{Morse flow Y-trees}

A \emph{Morse flow Y-tree} with asymptotic conditions $\tup{\cP{p}}$ consists of a basepoint $m_0 \in M$ and three half-infinite gradient flow lines: for $i = 1, 2, 3$,
\[
\chi_i \colon [0,\infty) \to M, \qquad
\frac{d\chi_i}{ds} = -\grad_g(f_i - f_{i-1})(\chi_i),
\]
satisfying
\[
\chi_i(0) = m_0, \qquad \lim_{s \to \infty} \chi_i(s) = \cP{p}_i.
\]
We denote by $\M(M, g; \tup{f}, \tup{\cP{p}})$ the moduli space of Morse flow Y-trees with asymptotic conditions $\tup{\cP{p}}$. Its virtual dimension is
\[
\op{virdim}\, \M(M, g;\, \tup{f},\, \tup{\cP{p}})
= \sum_{i=1}^3 \ind(\cP{p}_i) - 2n,
\]
where $\ind(\cP{p}_i)$ denotes the Morse index of $\cP{p}_i$ as a critical point of $f_i - f_{i-1}$.
By Proposition~\ref{prop: transversality of Morse trees}, for generic $\tup{f}$ and $g$, $\M(M, g; \tup{f}, \tup{\cP{p}})$ is a smooth manifold of dimension $\op{virdim}\, \M(M, g;\, \tup{f},\, \tup{\cP{p}})$.
In particular, if
\[
\sum_{i=1}^3 \ind(\cP{p}_i) = 2n,
\]
then $\M(M, g; \tup{f}, \tup{\cP{p}})$ is a smooth, compact, zero-dimensional manifold.

\subsubsection*{$J$-holomorphic discs}

Let $X = T^*M$ denote the cotangent bundle of $M$, equipped with the canonical symplectic form $\omega$. For each $\epsilon > 0$, define the Lagrangian submanifolds
\[
\Lambda_i^\epsilon = \op{graph}(\epsilon\,\dd f_i) \subset T^*M, \qquad i = 1, 2, 3,
\]
and set $\tup{\Lambda}^\epsilon = (\Lambda_1^\epsilon, \Lambda_2^\epsilon, \Lambda_3^\epsilon)$. A \emph{tuple of intersection points} $\tup{\cP{p}}^\epsilon = (\cP{p}_1^\epsilon, \cP{p}_2^\epsilon, \cP{p}_3^\epsilon)$ consists of points $\cP{p}_i^\epsilon \in \Lambda_{i-1}^\epsilon \cap \Lambda_i^\epsilon$ (indices mod $3$), which correspond canonically to the critical points $\cP{p}_i \in \crit(f_i - f_{i-1})$ via the identification $\Lambda_{i-1}^\epsilon \cap \Lambda_i^\epsilon \cong \crit(f_i - f_{i-1})$.

The metric $g$ induces a canonical almost complex structure $J$ on $X$: indeed, the Levi-Civita connection gives a splitting $T_{(m,\alpha)}X = T_mM \oplus T^*_mM$, and we define $J(v \oplus 0) = 0 \oplus \beta$ where $\beta(\cdot) = \langle v, \cdot\rangle_g$. The almost complex structure $J$ is compatible with $\omega$ in the sense that $\omega(\cdot, J\cdot)$ defines a Riemannian metric on $X$.

Let $\Sigma$ be the closed unit disk with three boundary marked points $z_1, z_2, z_3 \in \partial\Sigma$, arranged counterclockwise. Set $\dot\Sigma = \Sigma \setminus \{z_1, z_2, z_3\}$. Let $c_i$ be the open boundary arc of $\partial\Sigma \setminus \{z_1, z_2, z_3\}$ between $z_i$ and $z_{i-1}$. A \emph{$J$-holomorphic disc} with boundary conditions $\tup{\Lambda}^\epsilon$ and asymptotic conditions $\tup{\cP{p}}^\epsilon$ is a continuous map $u \colon \Sigma \to X$ satisfying
\[
\overline{\partial}_J u = 0 \quad \text{on } \dot\Sigma,
\qquad
u(z_i) = \cP{p}_i^\epsilon, \qquad
u(c_i) \subset \Lambda_i^\epsilon,
\]
where $\overline{\partial}_J u = \frac{1}{2}(\dd u + J \circ \dd u \circ j)$ and $j$ is the standard complex structure on $\Sigma$. Two such maps are equivalent if they differ by a biholomorphism of $(\Sigma, \{z_1, z_2, z_3\})$. We denote by $\M(X, J; \tup{\Lambda}^\epsilon, \tup{\cP{p}}^\epsilon)$ the moduli space of equivalence classes of $J$-holomorphic discs.

\subsection{Main Result}
As $\epsilon \to 0$, the Lagrangians $\Lambda_i^\epsilon$ collapse to the zero section of $T^*M$, and the moduli space $\M(X, J; \tup{\Lambda}^\epsilon, \tup{\cP{p}}^\epsilon)$ degenerates to $\M(M, g; \tup{f}, \tup{\cP{p}})$ in a suitable Gromov sense (\cite{fukaya1997zeroloop}, \cite{ekholm2007morse}, and \cite{ruan2006fukayacategory}). For the case of collapsing to a Lagrangian submanifold in a general symplectic manifold, see \cite{cant2023adiabatic}. The main result of this paper goes in the opposite direction: for each Morse flow Y-tree $\tup{\chi} \in \M(M, g; \tup{f}, \tup{\cP{p}})$, we construct a $J$-holomorphic disc converging to $\tup{\chi}$ as $\epsilon \to 0$.

\begin{theorem}\label{thm: main}
+Let $M$ be a closed $n$-dimensional manifold. Let $\tup{f}$ be a Morse tuple, and let $g$ be a Riemannian metric on $M$ such that $(\tup{f}, g)$ is Morse--Smale. Let $\tup{\cP{p}}$ be a tuple of critical points with $\sum_i \ind(\cP{p}_i) = 2n$. Let $\tup{\chi} = (\chi_1, \chi_2, \chi_3) \in \M(M, g; \tup{f}, \tup{\cP{p}})$ be a Morse flow Y-tree with common vertex $m_0 = \chi_1(0) = \chi_2(0) = \chi_3(0)$. Assume Assumptions~\ref{assumption: flat} and~\ref{assumption: linear}. Then, for all sufficiently small $\epsilon > 0$, there exists at least one $J$-holomorphic disc $u^\epsilon \in \M(X, J; \tup{\Lambda}^\epsilon, \tup{\cP{p}}^\epsilon)$ such that $u^\epsilon$ converges to $\tup{\chi}$ as $\epsilon \to 0$.
\end{theorem}

We make two remarks on the conditions of Theorem~\ref{thm: main}.

\begin{remark}
The flatness assumption on $g$ near the vertex is a simplification adopted here as a proof of concept. It will be removed in the next paper of this series.
\end{remark}

\begin{remark}
To strengthen the conclusion from existence of at least one disc to a precise count, one needs the additional assumption that $\M(X, J; \tup{\Lambda}^\epsilon, \tup{\cP{p}}^\epsilon)$ is cut out transversely for all sufficiently small $\epsilon > 0$, which holds for generic $(\tup{f}, g)$. Under this assumption, the gluing construction can be shown to be both injective and surjective, yielding a cobordism between $\M(X, J; \tup{\Lambda}^\epsilon, \tup{\cP{p}}^\epsilon)$ and $\M(M, g; \tup{f}, \tup{\cP{p}})$. In the rigid case this implies that the algebraic count of $J$-holomorphic discs equals one.

This is a main feature of our approach: although it only yields an algebraic count, it requires a less sophisticated pre-gluing and Banach manifolds construction. 
\end{remark}

The subsequent paper in this series will generalize the above results to all Morse trees and to moduli spaces of arbitrary dimension.

\subsection{Outline of the Proof}

The main part of the proof of Theorem~\ref{thm: main} is a gluing construction that produces $J$-holomorphic discs from Morse flow Y-trees. Given a Morse flow Y-tree $\tup{\chi} = (\chi_1, \chi_2, \chi_3)$, the construction proceeds in two steps.

\subsubsection*{Pre-gluing}

For each $i$, we construct an \emph{edge model} $w_i \colon [0,\infty) \times [0,1] \to X$ by pushing the gradient flow line $\chi_i$ in the fiber direction via the Hamiltonian vector fields of $\epsilon f_i$ for time $t$ and of $\epsilon f_{i-1}$ for time $(1-t)$. These serve as approximate solutions to the $J$-holomorphic curve equation near each edge.

To glue the three edge models into a disc, we introduce a \emph{local model}: a $J_0$-holomorphic map $\tdlm \colon \dot\Sigma \to \C^n$ with linear Lagrangian boundary conditions and appropriate asymptotic conditions at the punctures. A large ball in $\C^n$ is identified with a neighborhood of $(m_0, 0) \in X$ of size $O(\epsilon)$.

Two families of perturbation parameters are also built into the approximate solution from the outset. The first family $\tdlv \in \parameterSpace_V = \oplus_i \parameterSpace_V^i$, where $\parameterSpace_V^i$ consists of sections proportional to $\alpha_E^i v$ for $v \in T_{m_0}\mathscr{D}_{\cP{p}_i}$ (with $\alpha_E^i$ a bump function near the $i$-th end), has dimension $2n$ and is designed to fill the cokernel of the linearized operator $D_\tdlm$ at the local model. The second family $\constantVector \in \parameterSpace_E = \oplus_i \parameterSpace_E^i$, where $\parameterSpace_E^i$ consists of sections proportional to $\alpha_V^i v$ for $v \in T_{m_0}^\perp\mathscr{D}_{\cP{p}_i}$ (with $\alpha_V^i$ supported near the edge cutoff at $s \sim s_0$), has dimension $n$ and fills the cokernel of the linearized operators $D_{w_i}$ at the edge models. The \emph{pre-glued map} $\app$ is obtained by gluing the local model, the edge models, and these perturbation parameters together via cutoff functions at two scales $s_0$ and $s_1$.

\subsubsection*{Solving the nonlinear system}

The pre-glued map $\app$ is not an exact solution. The residual $\overline{\partial}_J\app$ decouples into a vertex part $\eta$ and edge parts $\eta_i$, and we seek infinite-dimensional corrections $(\xi, \xi_1, \xi_2, \xi_3)$ solving the system
\begin{align*}
\overline{\p}\xi = (\Pi_V - 1)\eta, \qquad D_i\xi_i = (\Pi_{E,i} - 1)\eta_i,
\end{align*}
where $\Pi_V$ is the projection onto the vertex cokernel and $\Pi_{E,i}$ is the projection onto the $i$-th edge cokernel. For each fixed $(\tdlv, \constantVector)$, this system is solved by a contraction mapping argument (Theorem~\ref{thm:contraction}), using $\epsilon$-weighted Sobolev norms $W^{1,p}_{(\epsilon)}$ on the edge models to control the degeneration of the domain as $\epsilon \to 0$.

Two analytic issues arise, and they are treated separately. First, the local model is \emph{obstructed}: $D_\tdlm$ has a $2n$-dimensional cokernel (Proposition~\ref{prop:cokernel}), so the standard implicit function theorem does not apply directly. This finite-dimensional obstruction is treated by introducing an obstruction section and applying a degree argument, following the obstruction-bundle perspective of Hutchings--Taubes \cite{hutchings2009gluing}. Second, the adiabatic limit introduces $\epsilon$-dependent edge domains and norms, so the relevant Sobolev constants and nonlinear estimates must be controlled uniformly in $\epsilon$. This is handled separately by the $\epsilon$-weighted Sobolev norms and the estimates used in the contraction mapping argument.

\subsubsection*{The obstruction section and degree argument}

With $(\xi, \xi_1, \xi_2, \xi_3)$ determined as functions of $(\tdlv, \constantVector)$, the remaining obstruction to producing a genuine $J$-holomorphic disc is encoded by the \emph{obstruction section}
\[
\obstructionSection \colon B^{2n}(r') \times B^n(r') \to Y_1 \times Y_2 \times Y_3 \times Z_1 \times Z_2 \times Z_3,
\]
defined by $\obstructionSection(\tdlv, \constantVector) = (\Pi_{V,1}\eta, \Pi_{V,2}\eta, \Pi_{V,3}\eta, \Pi_{E,1}\eta_1, \Pi_{E,2}\eta_2, \Pi_{E,3}\eta_3)$, whose zeroes correspond to actual $J$-holomorphic discs. The linearization $\obstructionSection_0$ at the origin is represented by the identity matrix in natural bases, hence has degree one. We show that $\obstructionSection$ is homotopic to $\obstructionSection_0$ through nonvanishing maps on the boundary $\p(B^{2n}(r') \times B^n(r'))$ (Proposition~\ref{prop: homotopy}), and conclude by the degree argument that $\obstructionSection$ has at least one zero, producing the desired $J$-holomorphic disc.

\subsection{Comparison with Previous Works}

Fukaya--Oh \cite{fukaya1997zeroloop} and Ekholm \cite{ekholm2007morse} set up the pre-gluing carefully enough that the implicit function theorem applies directly and yields exactly one $J$-holomorphic disc. In contrast, our approach yields an algebraic count of one via a degree argument rather than a direct uniqueness statement. It is tempting to include $\constantVector$ and $\tdlv$ in a single contraction mapping to obtain a unique solution outright, but the contraction mapping conditions do not appear to be satisfied for all parameters simultaneously, and doing so would seem to require a substantially more refined pre-gluing of approximate solutions and ambient Banach manifolds.

A further feature of our approach, absent from \cite{fukaya1997zeroloop} and \cite{ekholm2007morse}, is the restriction that the variation vector field of the edge models integrate to zero along the $\{0\} \times [0,1]$ cord. This restriction is needed to make the estimates work, at the cost of making the edge models obstructed. 

Compared with the obstruction bundle gluing in Hutchings--Taubes \cite{hutchings2009gluing}, the two main differences are that we work in the adiabatic limit case, and that our edge models are half-infinite bands approaching to non-critical points, rather than doubly infinite cylinders approaching to critical points. A minor difference is that we use $W^{1,p}$-norms for $p > 2$ in place of the Morrey spaces used in \cite{hutchings2009gluing}; this is more standard, but comes at the cost of losing a canonical projection onto a subspace.

\subsection{Outline of the Paper}

In Section~\ref{section: morse trees}, we review the moduli space of Morse trees and show that for generic $\tup{f}$ and $g$ it is a smooth manifold of the expected dimension; our approach here differs substantially from that of \cite{fukaya1997zeroloop}. In Section~\ref{section: J-holomorphic discs}, we review the moduli space of $J$-holomorphic discs and introduce the $\epsilon$-dependent Sobolev norms. In Section~\ref{section: edge model}, we study the edge model. In Section~\ref{section: local model}, we study the local model at the vertex, deriving explicit formulas for the local model and for the cokernel elements of $D_\tdlm$. In Section~\ref{section: pre-gluing}, we carry out the pre-gluing construction. In Section~\ref{section: gluing}, we carry out the full gluing construction and prove Theorem~\ref{thm: main}.

\subsection{Acknowledgements}
We thank Garrett Alston, Tobias Ekholm, Kenji Fukaya, Yong-geun Oh, Weidong Ruan, and Yuan Yao for helpful discussions. We are especially grateful to Ko Honda for his interest in this project and for his invaluable input, including suggestions on the exponential decay estimate and the choices of obstruction bundle gluing setup. Ke Zhu is grateful for the hospitality of the Department of Mathematics at the University of Minnesota during his visits.

\section{Morse Trees}\label{section: morse trees}

\subsection{Moduli Spaces of Morse Trees}

Throughout this paper, all trees are finite, connected, and have no vertex of valency two.
Vertices of valency $1$ are called \emph{exterior vertices}; all other vertices are called
\emph{interior vertices}. Edges incident to exterior vertices are called \emph{exterior edges};
all other edges are called \emph{interior edges}. We denote by $E$ the set of edges and by $V$
the set of vertices of the tree.

For each vertex $v \in V$ we assign a point $|v|$, and for each edge $e \in E$ we assign an
interval $|e| = [0,1]$. Given a tree $T$, we define the associated topological space
\[
|T| = \Bigl(\coprod_{v \in V} \{|v|\}\Bigr) \coprod \Bigl(\coprod_{e \in E} |e|\Bigr)\Big/{\sim},
\]
where $\sim$ is the equivalence relation identifying, for each edge $e$ from $v$ to $v'$, the
left endpoint of $|e|$ with $|v|$ and the right endpoint with $|v'|$.

A \emph{metric} on a tree assigns a positive real number to each interior edge and $+\infty$ to
each exterior edge.

\begin{definition}
A \emph{ribbon tree} $T$ is a tree equipped with a cyclic ordering of the edges incident to each
interior vertex.
\end{definition}

Given a ribbon tree $T$, there exists an embedding $\iota \colon |T| \to \Sigma$, where $\Sigma$
is the unit disc $|z| \leq 1$ in $\C$, such that $\iota^{-1}(\partial\Sigma)$ is the set of
exterior vertices. Such an embedding is unique up to homotopy.

\begin{definition}
A \emph{rooted ribbon tree} $(T, v_1)$ is a ribbon tree with a distinguished exterior vertex.
\end{definition}

A rooted ribbon tree has its exterior vertices canonically ordered as $(v_1, \dots, v_k)$
according to their counterclockwise order on $\partial\Sigma$ under $\iota$, where $k$ is the
number of exterior vertices.

Let $M$ be a closed $n$-dimensional manifold.

\begin{definition}
A tuple $\tup{f} = (f_1, \dots, f_k)$ of smooth functions $f_i \colon M \to \R$ is called
\emph{Morse} if $f_a - f_b$ is a Morse function for all distinct indices
$a, b \in \{1, \dots, k\}$.
\end{definition}

\begin{remark}
This condition is slightly stronger than that in \cite{fukaya1997zeroloop}, which requires only
that $f_i - f_{i-1}$ is Morse for $i \in \{1, \dots, k\}$. The stronger condition simplifies
certain transversality arguments below.
\end{remark}

We say $\tup{\cP{p}} = (\cP{p}_1, \dots, \cP{p}_k)$ is a \emph{tuple of critical points} of
$\tup{f}$ if $\cP{p}_i \in \crit(f_i - f_{i-1})$.

Fix a rooted ribbon tree $(T, v_1)$ with $k \geq 1$ exterior vertices and a Morse tuple
$\tup{f}$. We assign to each exterior vertex $v_i$ the Morse function $f_i - f_{i-1}$, and to
each boundary component $c$ of $\partial\Sigma \setminus \iota(|T|)$ the function
$f_c := f_i$, where $c$ is the arc bounded by $v_i$ and $v_{i+1}$. For each edge $e$, we
designate the two boundary components $c$ and $c'$ of $\partial\Sigma \setminus \iota(|T|)$
appearing on either side of $e$ under $\iota$.

Let $g$ be a Riemannian metric on $M$, and denote by $\grad_f$ the gradient vector field of a
function $f \colon M \to \R$ with respect to $g$.

\begin{definition}\label{def: morse tree}
A \emph{Morse tree} is a tuple $\tup{\chi} = (T, v_1, \ell, \chi)$, where $(T, v_1)$ is a
rooted ribbon tree, $\ell$ is a metric on $T$, and $\chi \colon |T| \to M$ is a continuous map
such that:
\begin{enumerate}
  \item for each exterior vertex $v_i$, $\chi(|v_i|) = \cP{p}_i$;
  \item for each edge $e \in E$, choosing an orientation of $|e|$ and parametrizing it as
  $[0, \ell(e)]$ with coordinate $s$ — which determines an ordered pair of boundary components
  $(c, c')$ with $c$ the left component and $c'$ the right — the restriction
  $\chi_e := \chi|_{|e|}$ satisfies
  \begin{equation}\label{eqn: chi}
    \frac{\dd}{\dd s} \chi_e = -\grad_{f_c - f_{c'}}(\chi_e).
  \end{equation}
\end{enumerate}
\end{definition}

Note that reversing the orientation of $|e|$ changes the sign on both sides of
Equation~\eqref{eqn: chi}.

Two Morse trees $\tup{\chi}$ and $\tup{\chi}'$ are \emph{equivalent} if there exists an
isomorphism of rooted ribbon trees $T \to T'$ inducing a homeomorphism
$\theta \colon |T| \to |T'|$ with $\chi' = \chi \circ \theta$.

We define $\M(M, g; T, \tup{f}, \tup{\cP{p}})$, the \emph{moduli space of Morse trees of type $T$},
to be the set of all such $\tup{\chi}$ modulo this equivalence relation.

\begin{proposition}[Proposition~12.5 in \cite{fukaya1997zeroloop}]
For generic $\tup{f}$ and $g$, the moduli space $\M(M, g; T, \tup{f}, \tup{\cP{p}})$ is a smooth
manifold of dimension
\[
\sum_{i=1}^k \ind(\cP{p}_i) - (k-1)n + k - 3 + \sum_v (3 - \op{val}(v)),
\]
where $\ind(\cP{p}_i)$ is the Morse index of $\cP{p}_i$ as a critical point of $f_i - f_{i-1}$,
the second sum is over all interior vertices, and $\op{val}(v)$ is the valency of $v$.
\end{proposition}

We will give a proof of this result that generalizes to the graph case.

Define
\[
\M(M, g; \tup{f}, \tup{\cP{p}}) = \coprod_T \M(M, g; T, \tup{f}, \tup{\cP{p}}),
\]
where the disjoint union is taken over all equivalence classes of rooted ribbon trees with $k$
exterior vertices. This is a topological space, with a topology induced by allowing the lengths
of interior edges to tend to zero. In fact, one has the following theorem.

\begin{theorem}[Theorem~1.4 in \cite{fukaya1997zeroloop}]\label{thm: transversality of Morse trees}
For generic $\tup{f}$ and $g$, the moduli space $\M(M, g; \tup{f}, \tup{\cP{p}})$ is a smooth
manifold of dimension
\[
\sum_{i=1}^k \ind(\cP{p}_i) - (k-1)n + k - 3,
\]
where $\ind(\cP{p}_i)$ is the Morse index of $\cP{p}_i$ as a critical point of $f_i - f_{i-1}$.
\end{theorem}

\subsection{Fredholm Setup for Morse Trees}\label{section: fredholm set up for trees}

Let $(T, v_1)$ be a rooted ribbon tree, and fix an orientation $\mathfrak{o}_e$ for each edge.
Let $\tup{f}$ be a Morse tuple and $\tup{\cP{p}}$ a tuple of critical points. Let $\mathcal{C}$
be the space of continuous maps $\chi \colon |T| \to M$ such that:
\begin{enumerate}
  \item for each exterior vertex $v_i$, $\chi(|v_i|) = \cP{p}_i$, for $i=1, \dots, k$;
  \item for each edge $e$, the restriction $\chi_e \colon |e| \to M$ is smooth.
\end{enumerate}
Fix $p > 2$. Let $\mathcal{B}$ be the completion of $\mathcal{C}$ under the $W^{1,p}$-norm,
defined via an embedding $M \hookrightarrow \R^N$ for large $N$ and a reference metric $\ell_0$
on $T$. The space $\mathcal{B}$ is independent of both choices.

Define a bundle $\mathcal{E} \to \mathcal{B}$ by
$\mathcal{E}_{(\ell, \chi)} = L^p(\chi^* TM)$, where the $L^p$-norm is defined using $g$ and
the metric $\ell$ on $T$. Define a section $\mathfrak{L}$ of $\mathcal{E} \to \mathcal{B}$ by
\[
\mathfrak{L}(\ell, \chi) = \frac{\dd}{\dd s}\chi + \grad_{\tup{f}}(\chi),
\]
where $s$ is the coordinate on $|T|$ determined by $\mathfrak{o}$ and $\ell$, and
$\grad_{\tup{f}}(\chi)$ is the vector field along $\chi$ defined edgewise by
$\grad_{f_c - f_{c'}}(\chi_e)$ as in Definition~\ref{def: morse tree}. Note that
$\mathfrak{L}$ depends on the choice of orientations.

Denote by $D_{(\ell, \chi)} \colon T_{(\ell, \chi)}\mathcal{B} \to \mathcal{E}_\chi$ the
linearization of $\mathfrak{L}$ at $(\ell, \chi)$ with respect to the Levi-Civita connection
of $g$:
\[
D_{(\ell, \chi)} \colon \R^{\sharp V(T) -k -1} \oplus W^{1,p}(|T|, \chi^* TM) \to L^p(|T|, \chi^* TM),
\]
where ${\sharp V(T) -k -1}$ is the number of interior edges of $T$.

\begin{proposition}[\cite{fukaya1997zeroloop}Section 6]
$D_\chi$ is a Fredholm operator.
\end{proposition}

Writing $D_{(\ell, \chi)} = D_\ell + D_\chi$, the operator $D_\chi$ is given by
\[
D_\chi \xi = \nabla_s \xi(\chi) + \nabla_\xi \grad_{\tup{f}}(\chi).
\]

Let $\chi \in \mathfrak{L}^{-1}(0)$, and choose $q$ with $1/p + 1/q = 1$. Suppose the image of
$D_\chi$ is a proper subspace of $L^p(|T|, \chi^* TM)$. Then by the Hahn--Banach theorem there
exists $0 \neq \eta \in L^q(|T|, \chi^* TM)$ perpendicular to the image of $D_\chi$:
\[
\langle D_\chi \xi, \eta \rangle = 0 \quad \text{for all } \xi \in W^{1,p}(|T|, \chi^* TM).
\]
Integrating by parts gives
\begin{align*}
0 &= \sum_e \int_0^{\ell(e)} \left( \frac{\dd}{\dd s}\langle \xi_e, \eta_e \rangle_g\
   - \langle \xi_e, \nabla_s \eta_e \rangle_g
   + \langle \nabla_{\xi_e} \grad_{f_c - f_{c'}}, \eta_e \rangle_g\,\right) \dd s \\
  &= \sum_e \Bigl[\langle \xi_e, \eta_e \rangle_g\big|_0^{\ell(e)}\Bigr]
   + \sum_e \int_0^{\ell(e)} \langle \xi_e,
   -\nabla_s \eta_e + \nabla_{\eta_e} \grad_{f_c - f_{c'}} \rangle_g\,\dd s \\
  &= \sum_v \Bigl\langle \xi(v),\, \sum_{e \in E_v} \mathfrak{o}_e(v) \cdot \eta_e(v) \Bigr\rangle
   + \sum_e \int_0^{\ell(e)} \langle \xi_e,
   -\nabla_s \eta_e + \nabla_{\eta_e} \grad_{f_c - f_{c'}} \rangle_g\,\dd s,
\end{align*}
where $E_v$ is the set of edges incident to $v$ and $\mathfrak{o}_e(v) \in \{\pm 1\}$ is
determined by the boundary orientation of $|e|$. Taking $\xi$ to be compactly supported in the
interior of each edge yields
\begin{equation}\label{eqn: adjoint equation for each edge}
-\nabla_s \eta_e + \nabla_{\eta_e} \grad_{f_c - f_{c'}} = 0
\end{equation}
for each edge $e$. In particular $\eta_e$ is smooth, and the vertex terms then give
\begin{equation}\label{eqn: balancing condition}
\sum_{e \in E_v} \mathfrak{o}_e(v) \cdot \eta_e(v) = 0.
\end{equation}

The formal adjoint $D_\chi^* \colon L^q(|T|, \chi^* TM) \to W^{-1,q}(|T|, \chi^* TM)$ is
\[
D_\chi^* \eta = -\nabla_s \eta(\chi) + \nabla_\eta \grad_{\tup{f}}(\chi).
\]

\subsection{Transversality}

\begin{definition}
A Morse tree $\tup{\chi}$ is \emph{somewhere injective} if, for each edge $e$ along which
$\chi_e$ is nonconstant, there exists an open subset $U \subset |e|$ such that
\begin{equation}\label{eqn: somewhere injective}
\chi(U) \cap \chi(|T| \setminus U) = \emptyset.
\end{equation}
\end{definition}

\begin{remark}
For the purposes of proving transversality, it suffices to require
Equation~\eqref{eqn: somewhere injective} only for exterior edges, but we adopt the above
definition for simplicity. Somewhere injective condition is well-known in the contexts of $J$-holomorphic curves to establish transversality, but our definition in the Morse tree case and its application for transversality appear to be new. See also \cite{cohen2012morse} and \cite{fukaya1996morse} for transversality of the moduli spaces of Morse graph flows and their topological applications.   
\end{remark}

\begin{proposition}\label{prop: one dimensional embedding}
For a generic pair of Morse functions $f, h \colon M \to \R$ with
$\crit(f) \cap \crit(h) = \emptyset$, the set $\{m \in M \mid \dd f \wedge \dd h(m) = 0\}$ is
a smoothly embedded $1$-dimensional submanifold of $M$.
\end{proposition}

\begin{proof}
Let $B = C^\infty(M) \times C^\infty(M) \times \R \times M$, and let $E = p_1^*(T^*M) \to B$
where $p_1 \colon B \to M$ is the projection. The section
$\mathcal{L}(f, h, s, m) = s\,\dd f(m) + \dd h(m)$ of $E$ is transverse to the zero section,
so $\mathcal{L}^{-1}(0)$ is a smooth submanifold of $B$. By Sard's theorem, a generic pair
$(f, h)$ is a regular value of the projection
$p_2 \colon \mathcal{L}^{-1}(0) \to C^\infty(M) \times C^\infty(M)$.

For such a generic pair, the section $L(s, m) = s\,\dd f(m) + \dd h(m)$ of
$F = p_3^*(T^*M) \to \R \times M$ is transverse to the zero section, so
$L^{-1}(0) \subset \R \times M$ is a smooth $1$-dimensional submanifold. Consider the
projection \[p_4 = p_3|_{L^{-1}(0)} \colon L^{-1}(0) \to M.\]

\begin{claim}
$p_4$ is an embedding away from $p_4^{-1}(\crit(f))$.
\end{claim}
\begin{proof}[Proof of claim]
For any point $(s_0, m_0) \in L^{-1}(0)$ with $m_0 \notin \crit(f)$, choose a coordinate
chart $(u^1, \dots, u^n)$ of $M$ around $m_0$. Let
$\alpha(t) = (s(t), u^1(t), \dots, u^n(t))$ be a parametrization of $L^{-1}(0)$ around
$(s_0, m_0)$ with $\alpha'(0) \neq 0$. Then
\[
s(t)\sum_i\frac{\p f}{\p u^i}(u^1(t),\dots, u^n(t))\,\dd u^i
+ \sum_i \frac{\p h}{\p u^i}(u^1(t),\dots, u^n(t))\,\dd u^i = 0.
\]
Differentiating with respect to $t$ at $t = 0$ gives
\[
s'(0)\sum_i\frac{\p f}{\p u^i}(m_0)\,\dd u^i
+ \sum_{i,j}\left(s_0\frac{\p^2 f}{\p u^i \p u^j}(m_0)
+ \frac{\p^2 h}{\p u^i \p u^j}(m_0)\right)\frac{\dd u^j}{\dd t}(0)\,\dd u^i = 0.
\]
If $\dd p_4|_{(s_0,m_0)} = 0$, then $\frac{\dd u^j}{\dd t}(0) = 0$ for $j = 1, \dots, n$.
Since $\alpha'(0) \neq 0$, we have $s'(0) \neq 0$, which forces
$\frac{\p f}{\p u^i}(m_0) = 0$ for all $i$, contradicting $m_0 \notin \crit(f)$. Hence
$p_4$ is an immersion away from $\crit(f)$.

For injectivity, suppose $(s, m), (s', m) \in L^{-1}(0)$. From the equations
$s\,\dd f(m) + \dd h(m) = 0$ and $s'\,\dd f(m) + \dd h(m) = 0$, we get $s = s'$. Hence
$p_4$ is an embedding away from $\crit(f)$.
\end{proof}

The set $\gamma = \{m \in M \mid \dd f \wedge \dd h(m) = 0\}$ satisfies
$\gamma \setminus \crit(f)$ is an embedded submanifold, and by symmetry so is
$\gamma \setminus \crit(h)$. Since $\crit(f) \cap \crit(h) = \emptyset$, $\gamma$ is an
embedded $1$-dimensional submanifold of $M$.
\end{proof}

\begin{definition}\label{def: regular}
A tuple $\tup{f}$ is \emph{regular} if:
\begin{enumerate}
  \item $\tup{f}$ is Morse;
  \item for any indices $a, b, c, d \in \{1, \dots, k\}$ with $a \neq b$, $c \neq d$, and
  $\{a,b\} \neq \{c,d\}$:
  \begin{enumerate}
    \item $\crit(f_a - f_b) \cap \crit(f_c - f_d) = \emptyset$;
    \item the sections $s\,\dd(f_a - f_b) + \dd(f_c - f_d)$ and
    $\dd(f_a - f_b) + s\,\dd(f_c - f_d)$ of $T^*M \to \R \times M$ are both transverse to the
    zero section.
  \end{enumerate}
\end{enumerate}
\end{definition}

\begin{lemma}
A generic tuple $\tup{f}$ is regular.
\end{lemma}

\begin{proof}
This follows from the proof of Proposition~\ref{prop: one dimensional embedding} by
intersecting finitely many full-measure subsets.
\end{proof}

\begin{proposition}\label{prop: regular implies injectivity}
Suppose $\tup{f}$ is regular and $n \geq 2$. Then for a generic metric $g$, every Morse tree
$\tup{\chi}$ of $(\tup{f}, g)$ is somewhere injective.
\end{proposition}

\begin{proof}
Since $\tup{f}$ is regular, the set
$\boldsymbol\gamma = \bigcup \{m \in M \mid \dd(f_a - f_b) \wedge \dd(f_c - f_d)(m) = 0\}$
(union over all admissible index quadruples as in Definition~\ref{def: regular}) is a union of
embedded $1$-dimensional submanifolds of $M$. For a generic metric $g$, the gradient
$\grad_{f_a - f_b}$ is tangent to $\boldsymbol\gamma$ at only finitely many points. If $e$ and
$e'$ are distinct edges with $\chi_e$ satisfying
$\frac{\dd}{\dd s}\chi_e = -\grad_{f_a - f_b}(\chi_e)$ and $\chi_{e'}$ satisfying
$\frac{\dd}{\dd s}\chi_{e'} = -\grad_{f_c - f_d}(\chi_{e'})$, then the vector fields
$\grad_{f_a - f_b}$ and $\grad_{f_c - f_d}$ are parallel along $\chi_e$ and $\chi_{e'}$ only
at finitely many points, which gives the required injectivity.
\end{proof}

\begin{proposition}\label{prop: transversality of Morse trees}
Let $\ell \geq 1$ be an integer.
For a $C^{\ell+1}$-generic tuple $\tup{f}$ and a $C^\ell$-generic metric $g$, $\mathfrak{L}$ is transverse to the
zero section at $\tup{\chi}$ for all $\tup{\chi} \in \M(M, g; \tup{f}, \tup{\cP{p}})$ and all
$\tup{\cP{p}}$.
\end{proposition}
\begin{proof}
Let $U$ be a small neighborhood of the union of critical point sets. Let $\mathcal{W}$ be the
space of tuples $\tup{h} = (h_1, \dots, h_k)$ of $C^{\ell+1}$-functions on $M$ vanishing on
$\overline{U}$. Fix a rooted ribbon tree $(T, v_1)$. Let $\mathcal{B}$ be the $W^{1,p}$-space
of maps from $|T|$ to $M$ sending exterior vertices to $\tup{\cP{p}}$, as defined in
Section~\ref{section: fredholm set up for trees}. Let $\mathcal{E}$ be the bundle over
$\mathcal{W} \times \mathcal{B}$ with fiber $\mathcal{E}|_{(\tup{h}, \ell, \chi)} =
L^p(\chi^* TM)$, and define a section $\mathbb{L}$ of $\mathcal{E} \to \mathcal{W} \times
\mathcal{B}$ by
\[
\mathbb{L}(\tup{h}, \ell, \chi) = \frac{\dd}{\dd s}\chi + \grad_{\tup{f} + \tup{h}}(\chi).
\]

We claim $\mathbb{L}$ is transverse to the zero section. Let $\mathbb{D}$ be the linearization
of $\mathbb{L}$ at $(0, \ell, \chi)$, where $s$ is the coordinate on $|T|$ with respect to
$\ell$. For any $(\tup{h}, \dot\ell, \xi) \in T_{(0,\ell,\chi)}(\mathcal{W} \times
\mathcal{B})$,
\[
\mathbb{D}(\tup{h}, \dot\ell, \xi)
= \nabla_s\xi(\chi) + \nabla_\xi\grad_{\tup{f}}(\chi)
+ \frac{\dd}{\dd s}\chi \cdot \dot\ell + \grad_{\tup{h}}(\chi).
\]
Suppose the image of $\mathbb{D}$ is a proper subspace of $L^p(\chi^* TM)$. By the
Hahn--Banach theorem, there exists $0 \neq \eta \in L^q(\chi^*TM)$ such that
$\langle \mathbb{D}(\tup{h}, \dot\ell, \xi), \eta\rangle = 0$ for all $(\tup{h}, \dot\ell,
\xi)$. This forces $\langle D_\chi\xi, \eta\rangle = 0$,
$\langle \frac{\dd}{\dd s}\chi \cdot \dot\ell, \eta\rangle = 0$, and
$\langle \grad_{\tup{h}}(\chi), \eta\rangle = 0$. The first equation implies that $\eta$
satisfies Equations~\eqref{eqn: balancing condition}
and~\eqref{eqn: adjoint equation for each edge}.

If $\eta$ is nonzero on some edge $e$ along which $\chi_e$ is nonconstant, then by
Proposition~\ref{prop: regular implies injectivity} there exists an open subset $V \subset |e|$
with $\chi_e(V) \cap \chi(|T| \setminus |e|) = \emptyset$ and $\chi_e(V) \cap U = \emptyset$.
For any $y \in V$, one can choose $\tup{h}$ supported near $m = \chi_e(y)$ with
$\grad_{\tup{h}}(m) = \eta(\chi_e(y))$, and $\dot\ell= \xi=0$, giving
$\langle \grad_{\tup{h}}(\chi), \eta\rangle > 0$, a contradiction. Hence $\eta$ vanishes on
all nonconstant edges. Since no two adjacent edges can both be constant by
Definition~\ref{def: regular}, solving Equations~\eqref{eqn: balancing condition}
and~\eqref{eqn: adjoint equation for each edge} for any constant edge forces $\eta$ to be
nonzero on some adjacent nonconstant edge, and the same argument yields a contradiction.

Hence $\mathbb{L}$ is transverse to the zero section, so $\mathbb{L}^{-1}(0)$ is a smooth
submanifold of $\mathcal{W} \times \mathcal{B}$. Let
$\op{pr} \colon \mathbb{L}^{-1}(0) \to \mathcal{W}$ be the projection. By Sard's theorem, a
generic $\tup{h} \in \mathcal{W}$ is a regular value of $\op{pr}$, and for such $\tup{h}$ the
section $\mathfrak{L}$ is transverse to the zero section.
\end{proof}

This completes the proof of Theorem~\ref{thm: transversality of Morse trees}.

\section{\texorpdfstring{$J$}{J}-holomorphic Discs}\label{section: J-holomorphic discs}

\subsection{Moduli Spaces of \texorpdfstring{$J$}{J}-holomorphic Discs}

Let $J$ be the canonical almost complex structure on $X = T^*M$ induced by $g$. Concretely,
for any $(m, \alpha) \in T^*M$, the Levi-Civita connection of $g$ gives a splitting
$T_{(m,\alpha)}X = T_mM \oplus T^*_mM$, and $J$ sends $v \oplus 0$ to $0 \oplus \beta$ where
$\beta(\cdot) = \langle v, \cdot\rangle_g$. The almost complex
structure $J$ is compatible with the canonical symplectic form $\omega$ in the sense that
$\omega(\cdot, J\cdot)$ defines a Riemannian metric on $X$.

For $\epsilon > 0$, let $\Lambda_i^\epsilon = \op{graph}(\epsilon\,\dd f_i) \subset T^*M$ and
$\tup{\Lambda}^\epsilon = (\Lambda_1^\epsilon, \dots, \Lambda_k^\epsilon)$. A
\emph{tuple of intersection points}
$\tup{\cP{p}}^\epsilon = (\cP{p}_1^\epsilon, \dots, \cP{p}_k^\epsilon)$ consists of points
$\cP{p}_i^\epsilon \in \Lambda_{i-1}^\epsilon \cap \Lambda_i^\epsilon$.

A \emph{$J$-holomorphic disc} is a pair $(\tup{z}, u)$, where
$\tup{z} = (z_1, \dots, z_k)$ is a tuple of distinct points on $\partial\Sigma$ ordered
counterclockwise and $u \colon \Sigma \to X$ is a continuous map satisfying
\[
u(z_i) = \cP{p}_i^\epsilon, \qquad
u(\p_i \dot{\Sigma}) \subset \Lambda_i^\epsilon, \qquad
\overline{\partial}_J u = 0,
\]
where $\p_i \dot{\Sigma}$ is the boundary arc of $\partial\Sigma \setminus \tup{z}$ between $z_i$ and
$z_{i+1}$, and \[\overline{\partial}_J u := \frac{1}{2}(\dd u + J \circ \dd u \circ j).\] Two
$J$-holomorphic discs $(\tup{z}, u)$ and $(\tup{z}', u')$ are \emph{equivalent} if there
exists a biholomorphism $\phi \colon \Sigma \to \Sigma$ with $\phi(\tup{z}) = \tup{z}'$ and
$u = u' \circ \phi$.

We denote by $\M(X, J; \tup{\Lambda}^\epsilon, \tup{\cP{p}}^\epsilon)$ the moduli space of
equivalence classes of $J$-holomorphic discs.

\begin{theorem}[\cite{fukaya1997zeroloop,ekholm2007morse}]\label{thm: main2}
For a generic choice of $\tup{f}$ and sufficiently small $\epsilon > 0$, the moduli space
$\M(X, J; \tup{\Lambda}^\epsilon, \tup{\cP{p}}^\epsilon)$ is a smooth manifold diffeomorphic to
$\M(M, g; \tup{f}, \tup{\cP{p}})$.
\end{theorem}

In this paper we prove Theorem~\ref{thm: main}, which is a weaker version of  Theorem~\ref{thm: main2} in the case $k = 3$ and
when both moduli spaces are zero-dimensional.

\subsection{Fredholm Setup for \texorpdfstring{$J$}{J}-holomorphic Discs}

Let $(\tup{z}, u) \in \M(X, J; \tup{\Lambda}^\epsilon, \tup{\cP{p}}^\epsilon)$, and let
$\dot\Sigma = \Sigma \setminus \tup{z}$. We define the $\epsilon$-dependent Sobolev norm on
sections of $u^*TX$ by
\begin{equation}\label{eqn: w1pepsilon norm}
\|\xi\|_{1,p,(\epsilon)}
= \left(\int_{\dot\Sigma} \epsilon^2 |\xi|^p + \epsilon^{2-p}|\nabla\xi|^p\right)^{1/p},
\end{equation}
where the integral uses the cylindrical volume form $\dd s \wedge \dd t$ near the ends of
$\dot\Sigma$. Similarly, the $\epsilon$-dependent $L^p$-norm on $(0,1)$-forms is
\begin{equation}\label{eqn: Lpepsilon norm}
\|\eta\|_{p,(\epsilon)}
= \left(\int_{\dot\Sigma} \epsilon^{2-p}|\eta|^p\right)^{1/p}.
\end{equation}
Note that $\|\cdot\|_{1,p,(\epsilon)}$ is equivalent to the standard $W^{1,p}$-norm, and
$\|\cdot\|_{p,(\epsilon)}$ is equivalent to the standard $L^p$-norm if we do not let $\eps\to 0$.  See Appendix for the geometric meanings of the weighted Sobolev norm.

The linearization of $\overline{\partial}_J$ at $(\tup{z}, u)$ in the direction of $u$ is the
operator \[D_u \colon W^{1,p}_{(\epsilon)}(u^*TX) \to L^p(\wedge^{0,1}\dot\Sigma \otimes_\C u^*TX)\]
given by
\begin{align}\label{formula: D}
D_u\xi
&= \tfrac{1}{2}(\nabla\xi + J\nabla\xi \circ j)
 - \tfrac{1}{2}J(\nabla_\xi J)\partial_J u \nonumber\\
&= \tfrac{1}{2}\bigl[(\nabla_s\xi + J\nabla_t\xi)
 - \tfrac{1}{2}J(\nabla_\xi J)(u_s - Ju_t)\bigr]
 \otimes_J (\dd s - \iu\,\dd t),
\end{align}
where $s + \iu t$ is a holomorphic coordinate on $\Sigma$.

\begin{lemma}
The operator $D_u$ is Fredholm.
\end{lemma}

\begin{proof}
For any $\epsilon > 0$, the norm $\|\cdot\|_{1,p,(\epsilon)}$ is equivalent to the norm at
$\epsilon = 1$, for which $D_u$ is known to be Fredholm.
\end{proof}

\section{Edge Models}\label{section: edge model}

\subsection{Asymptotic Properties}

Fix $\epsilon > 0$. Let $f_i \colon M \to \R$ be Morse functions for $i \in \{1,2,3\}$, let
$\cP{p}_i \in \crit(f_i - f_{i-1})$, and let $\chi_i \colon [0,\infty) \to M$ be a
half-infinite Morse flow line satisfying
\[
\frac{\dd}{\dd s}\chi_i - \grad_{\epsilon(f_i - f_{i-1})}(\chi_i) = 0,
\qquad
\lim_{s \to \infty}\chi_i(s) = \cP{p}_i,
\qquad
\chi_i(0) = m_0.
\]
Note that by abusing the notation with change the direction of $s$, so that it is more compatible with the $J$-holomorphic strip convention that we are familiar with.

Define
\begin{align}\label{eqn: w_i}
w_i \colon [0,\infty) \times [0,1] \to X,
\qquad
w_i(s,t) = \Phi_i(t)\bigl(\chi_i^\epsilon(s)\bigr), \quad i = 1,2,3,
\end{align}
where:
\begin{enumerate}
  \item $\chi_i^\epsilon(s) = \chi_i(\epsilon s)$,
  \item $\Phi_i(t) = \varphi_{1-t}^{\epsilon f_i} \circ \varphi_t^{\epsilon f_{i-1}}$,
  \item $\varphi_t^H \colon X \to X$ is the time-$t$ Hamiltonian flow of $H \colon X \to \R$.
\end{enumerate}
Then $w_i$ satisfies the Lagrangian boundary conditions
$w_i(\cdot, 0) \subset \Lambda_i^\epsilon$ and $w_i(\cdot, 1) \subset \Lambda_{i-1}^\epsilon$.
A direct computation following \cite[Formula~(5.4)]{fukaya1997zeroloop} gives the following.

\begin{lemma}\label{lemma: dbar J estimate for edge model}
There exists a constant $C$ independent of $\epsilon$ such that
\[
\|\overline{\partial}_J w_i\|_{L^p(M)} \leq C\epsilon^{2 - 1/p}.
\]
\end{lemma}

\subsection{Fredholm Properties}

We drop the subscript $i$ and write $\chi = \chi_i$, $\chi^\epsilon = \chi_i^\epsilon$,
$w = w_i$, $\Phi = \Phi_i$. We study vector fields along $w$ and the linearized operator
$D_w$.

Note that $\Phi(t)^{-1}(w(s,t)) = \chi^\epsilon(s)$ for all
$(s,t) \in [0,\infty) \times [0,1]$, so
\[
\dd\Phi^{-1}(t)(T_{w(s,t)}X) \subset T_{\chi^\epsilon(s)}X.
\]

Let $\bar\chi^\epsilon(s,t) = \chi^\epsilon(s)$ for $t \in [0,1]$. Define
$W^{1,p}_{(\epsilon), 0}((\bar\chi^\epsilon)^*TX)$ to be the space of sections $a$ satisfying:
\begin{enumerate}
  \item $a \in W^{1,p}((\bar\chi^\epsilon)^*TX)$;
  \item $a(s,0) \in T_{\chi^\epsilon(s)}M$ and $a(s,1) \in T_{\chi^\epsilon(s)}M$;
  \item $a(0,t) \in T_{\chi^\epsilon(0)}M$;
  \item $\int_0^1 a(0,t)\,\dd t = 0$.
\end{enumerate}
The norm is
\[
\|a\|_{W^{1,p}_{(\epsilon)}}
= \left(\int_0^\infty\int_0^1
  \epsilon^2|a|^p + \epsilon^{2-p}|\nabla a|^p\,\dd t\,\dd s\right)^{1/p}.
\]

Define $L^p_{(\epsilon)}((\bar\chi^\epsilon)^*TX)$ to be the $L^p$-space with norm
\[
\|b\|_{L^p_{(\epsilon)}}
= \left(\int_0^\infty\int_0^1 \epsilon^{2-p}|b|^p\,\dd t\,\dd s\right)^{1/p}.
\]

Define the operator
\[
D_{\bar\chi^\epsilon} \colon
W^{1,p}_{(\epsilon), 0}((\bar\chi^\epsilon)^*TX)
\to L^p_{(\epsilon)}((\bar\chi^\epsilon)^*TX)
\]
by
\[
D_{\bar\chi^\epsilon} a
= \nabla_s a + J(\bar\chi^\epsilon)\bigl(\nabla_t + \nabla X_{\epsilon(f_i - f_{i-1})}\bigr)a.
\]

Define $W^{1,p}_{(\epsilon), 0}(w^*TX)$ to be the space of sections $\xi$ satisfying:
\begin{enumerate}
  \item $\xi \in W^{1,p}(w^*TX)$;
  \item $\xi(s,0) \in T_{w(s,0)}\Lambda_i$ and $\xi(s,1) \in T_{w(s,1)}\Lambda_{i-1}$;
  \item $\xi(0,t) \in \dd\Phi(t)\,T_{\chi^\epsilon(0)}M$;
  \item $\int_0^1 \dd\Phi^{-1}(t)\,\xi(0,t)\,\dd t = 0$.
\end{enumerate}
The norm is
\[
\|\xi\|_{W^{1,p}_{(\epsilon)}}
= \left(\int_0^\infty\int_0^1
  \epsilon^2|\xi|^p + \epsilon^{2-p}|\nabla\xi|^p\,\dd t\,\dd s\right)^{1/p}.
\]

Define $L^p_{(\epsilon)}(w^*TX)$ with norm
\[
\|\eta\|_{L^p_{(\epsilon)}}
= \left(\int_0^\infty\int_0^1 \epsilon^{2-p}|\eta|^p\,\dd t\,\dd s\right)^{1/p}.
\]

Let $D_w \colon W^{1,p}_{(\epsilon), 0}(w^*TX) \to L^p_{(\epsilon)}(w^*TX)$ be defined by
\[
D_w = \dd\Phi \circ D_{\bar\chi^\epsilon} \circ \dd\Phi^{-1}(t),
\]
so that the following diagram commutes:
\begin{equation}\label{eqn: pushing to morse}
\begin{tikzcd}
W^{1,p}_{(\epsilon), 0}(w^*TX)
  \arrow[r, "D_w"]
  \arrow[d, "\dd\Phi^{-1}(t)"']
& L^p_{(\epsilon)}(w^*TX)
  \arrow[d, "\dd\Phi^{-1}(t)"] \\
W^{1,p}_{(\epsilon), 0}((\bar\chi^\epsilon)^*TX)
  \arrow[r, "D_{\bar\chi^\epsilon}"]
& L^p_{(\epsilon)}((\bar\chi^\epsilon)^*TX)
\end{tikzcd}
\end{equation}

The operator $D_w$ is related to the linearization of $\overline{\partial}_J$ as follows.

\begin{lemma}\label{lemma: linearization}
Let $\mathcal{L}_w$ denote the linearization of $\overline{\partial}_J$ at $w$. Then
\[
\overline{\partial}_J(\exp_w \xi)
= \overline{\partial}_J w + D_w\xi
+ O\bigl(\epsilon^2|\xi| + \epsilon|\nabla\xi| + \epsilon|\xi|^2 + |\xi|\cdot|\nabla\xi|\bigr).
\]
\end{lemma}

\begin{proof}
By the standard linearization formula for $\overline{\partial}_J$
(see e.g.\ \cite{mcduff2012j}),
\[
\overline{\partial}_J(\exp_w\xi)
= \overline{\partial}_J w + \mathcal{D}_w\xi
+ O\bigl(\epsilon|\xi|^2 + |\xi|\cdot|\nabla\xi|\bigr),
\]
where $\mathcal{D}_w$ is the linearization with respect to the Levi-Civita connection. The
comparison of $\mathcal{D}_w$ and $D_w$ is carried out in the proof of
Proposition~7.1 of \cite{fukaya1997zeroloop}: the pointwise difference is
$O(\epsilon^2|\xi| + \epsilon|\nabla\xi|)$.
\end{proof}

The half-infinite flow line $\chi$ extends uniquely to a complete flow line
$\tilde\chi \colon \R \to M$. Denote $\cP{p} = \lim_{s \to +\infty}\tilde\chi(s)$ and
$\cP{q} = \lim_{s \to -\infty}\tilde\chi(s)$. We say $\chi$ is \emph{transversely cut out} if
the descending manifold of $\cP{p}$ and the ascending manifold of $\cP{q}$ intersect
transversely at $\tilde\chi(s)$ for every $s \in [0,\infty)$.

\begin{proposition}\label{prop: fredholm index of edge}
The operator $D_w$ is Fredholm with
$\dim\ker D_w = 0$ and $\dim\op{coker} D_w = n - \ind(\cP{p})$.
\end{proposition}

\begin{proof}
Since $\dd\Phi(t)$ is an isomorphism, it suffices to prove the statement for
$D_{\bar\chi^\epsilon}$. Let $a \in \ker D_{\bar\chi^\epsilon}$. By the proof of
Proposition~6.1 in \cite{fukaya1997zeroloop}, $a_i$ takes values in $TM$,
and is independent of $t$. Thus $a_i(s,t)=a_i(s)$ satisfies the linearized Morse equation
\[
\nabla_s a + \nabla_a\grad_{\epsilon(f_i - f_{i-1})}(\chi^\epsilon) = 0.
\]
Hence $a$ is determined by $a(0) \in T_{m_0}M$. Since
$a \in W^{1,p}_{(\epsilon), 0}((\bar\chi^\epsilon)^*TX)$, we have $\lim_{s\to\infty}a(s) = 0$,
so $a(0)$ must be tangent to the descending manifold of $\cP{p}$. But the taming condition implies $a(0) = 0$. Hence $\ker D_{\bar\chi^\epsilon} = 0$.

For the cokernel, let $b \in L^q((\bar\chi^\epsilon)^*TX)$ satisfy
$\langle D_{\bar\chi^\epsilon} a, b\rangle = 0$ for all
$a \in W^{1,p}_{(\epsilon), 0}((\bar\chi^\epsilon)^*TX)$. By the same argument (it is actually harder than the kernel part, and there are typos in that part of \cite{fukaya1997zeroloop}, so we give a sketch of their argument: in \cite{fukaya1997zeroloop}
Proposition~6.1, using integration by parts to pair the cokernel element $b_i^\eps=(e_i^\eps,f_i^\eps)$ with $a_i^\eps$, the integration zero condition of $a_i^\eps$ along the boundary cord
$\{0\}\times[0,1]$ forces the cokernel horizontal component
$e_i^\eps(0,\cdot)$ in $T^*M$ to be a constant, which in turn forces
$\nabla_\tau f_i^\eps(0,t)=0$ for all $t\in[0,1]$, where $f_i^\eps$
is the cokernel vertical component. This shows
$\frac{d}{d\tau}\gamma_i^\eps(0)=0$,  where
$\gamma_i^\eps(\tau)=\frac{1}{2}\int_0^1|f_i^\eps(\tau,t)|^2\,dt$.  Then a maximum principle applied
to the ODE inequality of $\gamma_i^\eps(\tau)$ implies
$\gamma_i^\eps(\tau)\equiv 0$, so $ f_i^\eps\equiv 0$.
Then it implies $\nabla_\tau f_i^\eps\equiv 0$, and by the relation of $e_i^\eps$
to $f_i^\eps$, $\nabla_t e_i^\eps\equiv 0$, so the cokernel element $b_i$
is $t$-independent). 
Thus $b$ takes values in $TM$, is independent of $t$, and satisfies
\begin{equation}\label{eqn: adjoint equation for edge}
  -\nabla_s b + \nabla_b\grad_{\epsilon(f_i - f_{i-1})}(\chi^\epsilon) = 0.
\end{equation}
Since $b \in L^q((\bar\chi^\epsilon)^*TX)$, we have $\lim_{s\to \infty}b(s) = 0$, so $b(0)$
must be perpendicular to the descending manifold of $\cP{p}$. Thus
$\dim\op{coker} D_{\bar\chi^\epsilon} = n - \ind(\cP{p})$.
\end{proof}

\section{Vertex Models}\label{section: local model}

\subsection{Asymptotic Behavior}

Given three vectors $c_1, c_2, c_3 \in \R^n$, define Lagrangian subspaces
$\td{\tup\Lambda} = (\td\Lambda_1, \td\Lambda_2, \td\Lambda_3)$ in $\C^n$ by
\[
\td\Lambda_i = \{x + \iu c_i \mid x \in \R^n\}.
\]
Define paths $\gamma_i \colon [0,1] \to \R^n$ by $\gamma_i(t) = t(c_{i+1} - c_i) + c_i$, and
let $\Sigma$ denote the unit disk in $\C$.

A \emph{local model} is a pair $(\tup{z}, \tdlm)$ consisting of:
\begin{enumerate}
  \item three distinct boundary marked points $\tup{z} = (z_1, z_2, z_3)$ on $\partial\Sigma$,
  ordered counterclockwise;
  \item for each $i$, a biholomorphism $\phi_i \colon [0,\infty) \times [0,1] \to V_i\setminus z_i$
  for some open neighborhood $V_i$ of $z_i$ in $\Sigma$;
  \item a smooth map $\tdlm \colon \Sigma \setminus \tup{z} \to \C^n$ satisfying
  \[
  \overline{\partial}\tdlm = 0, \qquad
  \tdlm(\partial_i\dot\Sigma) \subset \td\Lambda_i, \qquad
  \lim_{s\to\infty} \operatorname{Im}(\tdlm(\phi_i(s,t))) = \gamma_i(t),
  \]
  where $\partial_i\dot\Sigma$ is the boundary arc of $\partial\Sigma \setminus \tup{z}$
  between $z_i$ and $z_{i+1}$.
\end{enumerate}
Two local models $(\tup{z}, \tdlm)$ and $(\tup{z}', \tdlm')$ are \emph{equivalent} if there exists
a biholomorphism $\psi \colon \Sigma \to \Sigma$ with $\psi(\tup{z}) = \tup{z}'$ and
$\tdlm = \tdlm' \circ \psi$.

Note that $\tdlm$ can be translated by any constant $v \in \R^n \subset \C^n$ to give another
local model. In fact, the following holds.

\begin{proposition}[Proposition~4.1 in \cite{fukaya1997zeroloop}]
The moduli space of local models is diffeomorphic to $\R^n$.
\end{proposition}

The local models can be written down explicitly. Define the linear map
$\td{L}_i \colon [0,\infty) \times [0,1] \to \C^n$ by
\begin{equation}\label{eqn: linear model}
\td{L}_i(s,t) := (s + \iu t)(c_{i+1} - c_i) + \iu c_i.
\end{equation}

\begin{lemma}[Section~7.1 in \cite{ruan2006fukayacategory}]\label{lemma: local model asymptotic property}
There exists a unique local model $\tdlm$ and coordinates $\phi_i$ such that on each end,
\[
\left|\tdlm(s,t) - \td{L}_i(s,t)\right| \leq Ce^{-\pi s}.
\]
\end{lemma}

\subsection{Fredholm Theory}

Let $0 < \delta < 1$ and $\dot\Sigma = \Sigma \setminus \tup{z}$. Define
$W^{1,p}_\delta(\tdlm^*T(\C^n; \td{\tup\Lambda}))$ to be the space of
$W^{1,p}_\delta$-sections $\xi$ of $\tdlm^*T\C^n$ satisfying:
\begin{enumerate}
  \item The $W^{1,p}_\delta$-norm is computed with respect to a fixed $j$-compatible Riemannian
  metric $g_{\dot\Sigma}$ that is cylindrical near each end of $\dot\Sigma$, with exponential
  weight $e^{\delta s}$.
  \item For each $z \in c_i$, the section $\xi(z)$ lies in $T_{\tdlm(z)}\td\Lambda_i$.
\end{enumerate}
Let $L^p_\delta(\wedge^{0,1}\dot\Sigma \otimes_\C \tdlm^*T\C^n)$ be the corresponding
$\delta$-weighted $L^p$-space. The linearized $\overline{\partial}$-operator is
\[
\overline{\partial} \colon
W^{1,p}_\delta\bigl(\tdlm^*T(\C^n; \td{\tup\Lambda})\bigr)
\to L^p_\delta\bigl(\wedge^{0,1}\dot\Sigma \otimes_\C \tdlm^*T\C^n\bigr).
\]

\begin{proposition}\label{prop: computation of cokernel}
The operator $\overline{\partial}$ is an injective Fredholm operator of index $-2n$. In
particular, $\ker\overline{\partial} = 0$ and $\dim\op{coker}\overline{\partial} = 2n$.
\end{proposition}

\begin{proof}
The Fredholm property follows from the standard theory of elliptic operators on manifolds with
cylindrical ends (see \cite{robbin2001asymptotic} or Chapter~3 of \cite{donaldson2002floer}).
The weight $\delta > 0$ is chosen smaller than the spectral gap $\pi$ of the operator
$J_0\frac{\partial}{\partial t}$ on $W^{1,2}([0,1], \partial[0,1]; \C^n, \R^n)$.

Let $\xi \in \ker\overline{\partial}$. The totally real boundary conditions allow us to apply
the Schwarz reflection principle (cf.\ Chapter~13 of \cite{conway1995functions}) to extend
$\xi$ to a holomorphic function on $\C$ by setting $\xi(z) = \overline{\xi(1/\bar z)}$ outside
the open unit disk. Since $\xi \in W^{1,p}_\delta$, the Sobolev embedding
\[
  W^{1,p}([s-1, s+1] \times [0,1]) \hookrightarrow C^0([s-1, s+1] \times [0,1])
\]
gives $|\xi(s,t)| \leq C_B e^{-\delta s}\|\xi\|_{1,p,\delta}$, so $\xi$ decays exponentially
at each puncture. The punctures are thus removable, and $\xi$ extends to an entire holomorphic
function on $\C$.

Writing $\xi = u + \iu v$ with $u, v \colon \C \to \R^n$, the imaginary part $v$ is harmonic
on $\C$ and vanishes on $\partial\Sigma$. By the maximum principle, $v \equiv 0$, so $u$ is
constant. Thus $\ker\overline{\partial} = 0$.

We outline the Fredholm index calculation. The totally real boundary conditions along each
puncture agree at infinity. Changing the weight $\delta$ from a small positive to a small
negative number increases the index by $n$ per end, by the Atiyah--Singer index theorem, or Lockart-McOwen index change formula (\cite{lockhart1985elliptic}). For a
sufficiently small negative weight, a straightforward comparison of the asymptotics of the
kernel and cokernel elements shows that the index equals that of the same operator on the disk
without punctures, which is $n\chi(\Sigma) = n$ since the Maslov index is zero. Since there
are three ends, the index of $\overline{\partial}$ at positive weight is $n - 3n = -2n$.
\end{proof}

\begin{remark}
Lemma~6.1 of \cite{fukaya1997zeroloop} contains a related statement, but the operator
considered there is not Fredholm.
\end{remark}

Now we modify the Fredholm setup so that $\overline{\partial}$ becomes surjective. We introduce
special elements that fill the cokernel.

\begin{assumption}[Rigid]\label{assumption: rigid Y tree}
$\displaystyle\sum_{i=1}^3 \ind(\cP{p}_i) = 2n$.
\end{assumption}

\begin{assumption}[Transverse intersection]\label{assumption: transversal intersection}
The descending manifolds $\mathscr{D}_{\cP{p}_1}, \mathscr{D}_{\cP{p}_2}, \mathscr{D}_{\cP{p}_3}$ intersect
transversely at $m_0$, i.e.,
\[
\mathscr{D}_{\cP{p}_1} \times \mathscr{D}_{\cP{p}_2} \times \mathscr{D}_{\cP{p}_3}
\pitchfork_{M \times M \times M} \Delta_M
\]
at $(m_0, m_0, m_0)$, where $\Delta_M \subset M^3$ is the diagonal.
\end{assumption}
By Proposition~\ref{prop: transversality of Morse trees}, for generic data, Assumption~\ref{assumption: transversal intersection} can be met.

Let $\{e_i^1, \dots, e_i^{\ind(\cP{p}_i)}\}$ be an orthogonal basis of
$T_{m_0}\mathscr{D}_{\cP{p}_i}$, and set
$\tdlv = \tdlv_1 + \tdlv_2 + \tdlv_3$ where $\tdlv_i = \sum_{j=1}^{\ind(\cP{p}_i)}\alpha_V^i \varsigma_i^j e_i^j$ with
$\varsigma_i^j \in \R$, and $\alpha_V^i$ is the bump
function on the $i$-th end of $\dot\Sigma$ such that it equals $1$ when $s \geq 2$ and it equals $0$ when $s \leq 1$  as  illustrated in
Figure~\ref{fig: cutoff functions}.
 Under Assumptions~\ref{assumption: rigid Y tree}
and~\ref{assumption: transversal intersection},
\[
T_{m_0}\mathscr{D}_{\cP{p}_1} \cap T_{m_0}\mathscr{D}_{\cP{p}_2} \cap T_{m_0}\mathscr{D}_{\cP{p}_3} = \{0\}.
\]

\begin{proposition}\label{prop: add elements to fill cokernel}
The map
\begin{align*}
\overline{\partial} \colon
W^{1,p}_\delta\bigl(\tdlm^* T(\C^n; \td{\tup{\Lambda}})\bigr)
\oplus \bigoplus_{i=1}^3 \R^{\ind(\cP{p}_i)}
&\to L^p_\delta(\wedge^{0,1}\dot\Sigma \otimes_\C \tdlm^* T\C^n), \\
(\xi,\, \varsigma_1^1, \dots, \varsigma_1^{\ind(\cP{p}_1)},\,
\varsigma_2^1, \dots, \varsigma_2^{\ind(\cP{p}_2)},\,
\varsigma_3^1, \dots, \varsigma_3^{\ind(\cP{p}_3)})
&\mapsto \overline{\partial}\xi + \overline{\partial}\tdlv
\end{align*}
is an isomorphism.
\end{proposition}

\begin{proof}
We first show that
\[
\frac{\dd \alpha_E^1}{\dd s}e_1^1, \dots, \frac{\dd \alpha_E^1}{\dd s} e_1^{\ind(\cP{p}_1)},\;
\frac{\dd \alpha_E^2}{\dd s}e_2^1, \dots, \frac{\dd \alpha_E^2}{\dd s} e_2^{\ind(\cP{p}_2)},\;
\frac{\dd \alpha_E^3}{\dd s}e_3^1, \dots, \frac{\dd \alpha_E^3}{\dd s} e_3^{\ind(\cP{p}_3)}
\]
are linearly independent in the quotient
\[
L^p_\delta(\wedge^{0,1}\dot\Sigma \otimes_\C \tdlm^*T\C^n)
\big/ \overline{\partial}\bigl[W^{1,p}_\delta(\tdlm^*(T(\C^n; \td{\tup{\Lambda}})))\bigr].
\]
Note that $\overline{\partial}(\beta_i e_i^j) = \beta'_i e_i^j \otimes (\dd s + \iu\,\dd t)$.
Suppose constants $\varpi_i^j$ satisfy
\[
\sum_{i,j} \varpi_i^j \beta'_i e_i^j
\in \overline{\partial}\bigl[W^{1,p}_\delta(\tdlm^*(T(\C^n; \td{\tup{\Lambda}})))\bigr].
\]
Then there exists
$\xi' \in W^{1,p}_\delta(\tdlm^*(T(\C^n; \td{\tup{\Lambda}})))$
with $\overline{\partial}\xi' = \sum_{i,j}\varpi_i^j \beta'_i e_i^j$.
Set $\td{\xi} = \sum_{i,j}\varpi_i^j \beta_i e_i^j$, so that
$\overline{\partial}(\td{\xi} - \xi') = 0$. Since both $\td{\xi}$ and $\xi'$ are bounded and
$\td{\xi} - \xi'$ satisfies the totally real boundary condition on $\partial\dot\Sigma$, the
maximum principle gives a constant vector $e_0 \in \R^n$ with $\td{\xi} - \xi' \equiv e_0$.
Since $\xi'$ decays exponentially on each end and
$T_{m_0}\mathscr{D}_{\cP{p}_1} \cap T_{m_0}\mathscr{D}_{\cP{p}_2} \cap T_{m_0}\mathscr{D}_{\cP{p}_3} = \{0\}$,
we conclude $e_0 = 0$ and hence $\varpi_i^j = 0$ for all $i, j$.

Since $\dim\op{coker}\overline{\partial} = 2n = \sum_i\ind(\cP{p}_i)$, these elements span the
cokernel, and the extended operator is an isomorphism.
\end{proof}

We also give an explicit description of the cokernel elements, which is used in the gluing
construction. Model the three-punctured disk $\dot\Sigma$ as the \emph{slit domain}
$D = \R \times [-1,1]$ with a cut along $[0,\infty) \times \{0\}$. The three strip-like ends
correspond to $s \to +\infty$ on the upper and lower sides of the cut and $s \to -\infty$. (We can also write down the cokernel elements explicitly in the 3-punctured disk domain, but the expressions are not as simple as on the slit domain).   

The formal adjoint of $\overline{\partial}$ is
$\overline{\partial}^* = -\partial_s + \iu\,\partial_t$, acting as
\[
\overline{\partial}^* \colon
W^{1,q}_{-\delta}(\dot\Sigma; \wedge^{0,1} \otimes \C^n)
\to L^q_{-\delta}(\dot\Sigma; \C^n),
\]
where $1/p + 1/q = 1$. The cokernel of $\overline{\partial}$ is canonically identified with
$\ker(\overline{\partial}^*) \subset L^q_{-\delta}$, with totally real boundary conditions (\cite{fukaya1997zeroloop}, or Appendix C in \cite{mcduff2012j}) on the three boundary lines of the slit, which has real dimension $2n$ by the Fredholm index calculation and $\ker \overline{\p} = \{0\}$.

\begin{proposition}[Explicit generators of $\ker\overline{\partial}^*$]\label{prop:cokernel}
For $0 < \delta < \pi$, a basis of $\ker(\overline{\partial}^*) \subset L^q_{-\delta}$ is given
by the following two families.

\medskip\noindent\textbf{Type~1} (constant modes):
\[
\rho^{(1)} = v \otimes (\dd s - \iu\,\dd t),
\]
where $v \in \R^n \subset \C^n$ is a constant vector.

\medskip\noindent\textbf{Type~2} (meromorphic modes):
\[
\rho^{(2)} = v \otimes \frac{e^{-\pi(s-\iu t)}}{1 - e^{-\pi(s-\iu t)}}
(\dd s - \iu\,\dd t),
\]
where $v \in \R^n \subset \C^n$ is a constant vector.
\end{proposition}

On the ends going to $+\infty$, the factor
\[
\frac{e^{-\pi(s-\iu t)}}{1-e^{-\pi(s-\iu t)}}
\]
decays like $e^{-\pi s}$. On the end going to $-\infty$, set
$z=s-\iu t$. Since $|e^{\pi z}|<1$ on this end, the Fourier expansion gives
\[
\frac{e^{-\pi z}}{1-e^{-\pi z}}
=
\frac{1}{e^{\pi z}-1}
=
-\frac{1}{1-e^{\pi z}}
=
-1-\sum_{m=1}^\infty e^{m\pi z}.
\]
Thus the Type~2 mode approaches the constant mode $-1$, with exponentially
decaying corrections of order $e^{\pi s}$ as $s\to-\infty$.

\begin{proof}
For Type~1,
\[
\overline{\partial}^*(v \otimes (\dd s - \iu\,\dd t))
=
(-\partial_s + \iu\,\partial_t)v
=0,
\]
since $v$ is constant. Using the positive cylindrical coordinate on each end,
constants lie in the negatively weighted space $L^q_{-\delta}$, so
$\rho^{(1)} \in L^q_{-\delta}$.

For Type~2, write
\[
f(z)=\frac{e^{-\pi z}}{1-e^{-\pi z}}, \qquad z=s-\iu t.
\]
Away from its pole, $f$ is holomorphic as a function of $z$, and hence
\[
(-\partial_s+\iu\,\partial_t)f
=-(\partial_s-\iu\,\partial_t)f
=0.
\]
Thus $\overline{\partial}^*\rho^{(2)}=0$. On the two ends with
$s\to+\infty$, $f=O(e^{-\pi s})$. On the end with $s\to-\infty$, the
expansion above gives $f=-1+O(e^{\pi s})$. Equivalently, in the positive
cylindrical coordinate $\sigma=-s$ on this end,
$f=-1+O(e^{-\pi\sigma})$. Hence $\rho^{(2)} \in L^q_{-\delta}$ for
$0<\delta<\pi$.

Both types of forms take real values on tangent vectors along the boundary
lines of the slit, where $t=-1,0,1$, so they satisfy the totally real boundary
conditions.

Linear independence follows from the distinct asymptotic behaviors of the two
types, and together the $2n$ forms match
$\dim\op{coker}\overline{\partial}=2n$.
\end{proof}

\begin{remark}
Type 1 and type 2 modes can be transformed to each other under conformal transforms of the slit model permuting the three cylindrical ends.    
\end{remark}
\section{Pre-gluing}\label{section: pre-gluing}

In this section we adopt the gluing profile of Hutchings--Taubes \cite{hutchings2009gluing}.
Let $\tup{\chi}$ be a Morse flow Y-tree with $k = 3$, and let $m_0 \in M$ denote the common
vertex
\[
m_0 = \chi_i(0), \qquad i = 1, 2, 3.
\]
This determines edge models $\tup{w}$ and a local model $\tdlm$. Gluing these together produces
an approximate solution. On the edge side, the gluing region lies near the boundary segment
$\{0\} \times [0,1] \subset [0,\infty) \times [0,1]$ (to be specified precisely below), and on
the local model side it lies near the ends of $\dot\Sigma$.

\subsection{Local Coordinates}

\begin{assumption}\label{assumption: flat}
The Riemannian metric $g$ on $M$ is flat near $m_0$.
\end{assumption}
This implies that the almost complex structure on $X$ is integrable near $x_0 = (m_0, 0)$.
We choose a coordinate $(u^1, u^2, \dots, u^n)$ of $M$ centered at $m_0$ with respect to which $g_{ij} = \delta_{ij}$. Different choices of coordinates differ by an $O(n)$ change of variables. 
Let $\varphi = (u^1, u^2, \dots, u^n, v^1, v^2, \dots, v^n): U \subset X \to \R^{2n}$ be the associated coordinate of $X$.
We identify $U$ with an open subset of $\C^n$ by $$x \mapsto \epsilon^{-1} \varphi(x),$$
sending $x_0$ to the origin, $M$ to $\R^n$, and $T^*_{m_0} M$ to $\iu \R^n$.
\begin{assumption}\label{assumption: linear}
We assume with respect to the coordinates $(u^1, u^2, \dots, u^n)$, the Morse functions $f_1, f_2, f_3$ are linear, i.e., does not contain terms of order greater or equal to two, in this neighborhood of $m_0$.
\end{assumption}
Assumption~\ref{assumption: linear} implies that the Lagrangians $\Lambda_1^\epsilon \cap U, \Lambda_2^\epsilon \cap U, \Lambda_3^\epsilon \cap U$ are mapped into $\iu \tup{c}_1 + \R^n, \iu \tup{c}_2 + \R^n, \iu \tup{c}_3 + \R^n \subset \C^n$,
where $$\tup{c}_i = (\frac{\p f_i}{\p u^1}(0), \frac{\p f_i}{\p u^2}(0), \dots, \frac{\p f_i}{\p u^n}(0)).$$

Let $\tdlm$ denote the local model from Lemma~\ref{lemma: local model asymptotic property} with the above choice of $\tup{c}_i$.


\subsection{Choice of Constants}
Let $0 < \delta < \frac{1}{2}$ be the decay rate in the weighted Sobolev spaces.
Recall that $p > 2$ is the Sobolev exponent, and we choose $p$ large so that 
\begin{equation}\label{eq: p}
p > 4.
\end{equation}
We choose $\sigma$ satisfying
\[
\frac{2}{p} < \sigma < 1 - \frac{2}{p}.
\]

Set
\begin{align}
s_0 &= \frac{p}{\delta}\left(1 - \frac{2}{p} - \sigma\right)\ln\epsilon^{-1} > 0,
\label{eq:s0}\\
s_1 &= \frac{p}{\delta}\left(1 - \frac{2}{p} + \sigma\right)\ln\epsilon^{-1} > 0,
\label{eq:s1}
\end{align}
and choose $h = \frac{1}{4}s_0$.

Let $w_i$ be defined by \eqref{eqn: w_i}, and through the chart $U$, we view $$w_i: [0, \epsilon^{-1}) \times [0,1] \to \C^n.$$
\begin{lemma}\label{lemma: asymptotics of w_i}
There exists a constant $C$ independent of $\epsilon$ such that
\begin{equation}
|\td{w}_i(s,t) - \td{L}_i(s,t)|_{\C^n} \leq C\epsilon(s^2 + 1)
\end{equation}
for all $0 \leq s \leq \epsilon^{-1}$.
\end{lemma}

\begin{proof}
Consider $F(\epsilon, s, t) = \epsilon^{-1}\Phi_i(t)\chi_i(\epsilon s) - \td{L}_i(s,t)$. Then
$F(0,0,0) = 0$, $\frac{\p F}{\p\epsilon}(0,0,0) = 0$, and
$\frac{\p F}{\p s}(0,0,0) = \frac{\p F}{\p t}(0,0,0) = 0$. All second-order partial
derivatives of $F$ are bounded by $C\epsilon$ for some constant $C$, and the lemma follows
from the Taylor expansion of $F$ at $(0,0,0)$.
\end{proof}

\subsection{Cutoff Functions}

Define a cutoff function $\beta_E \colon \dot\Sigma \to [0,1]$ as follows. On the $i$-th end
$S_i  = [0, \infty) \times [0,1]$, require:
\begin{enumerate}
\item $\beta_E(s) = 0$ for $s \leq s_0 - 2h$;
\item $\beta_E(s) = 1$ for $s \geq s_0$;
\item $\left|\frac{\dd\beta_E}{\dd s}(s)\right| < \frac{1}{h}$;
\item $\left|\frac{\dd\beta_E}{\dd s}(s)\right| = \frac{1}{2h}$ for all $s_0-\frac{3h}{2} \leq s \leq s_0 - \frac{h}{2}$.
\end{enumerate}
Extend $\beta_E$ by constants on the remainder of $\dot\Sigma$.

Similarly, define $\beta_V \colon \dot\Sigma \to \R$ by requiring on each end $S_i$:
\begin{enumerate}
\item $\beta_V(s) = 0$ for $s \geq s_1 + 2h$;
\item $\beta_V(s) = 1$ for $s \leq s_1$;
\item $\left|\frac{\dd\beta_V}{\dd s}(s)\right| < \frac{1}{h}$;
\item $\left|\frac{\dd\beta_V}{\dd s}(s)\right| = \frac{1}{2h}$ for all $s_1 + \frac{h}{2} \leq s \leq s_1 + \frac{3h}{2}$.
\end{enumerate}
In particular, this implies 
\[
\frac{1}{2} h^{\frac{1}{p} - 1}\leq \left(\int_{\dot\Sigma}\left| \frac{d \beta_E}{d s} \right|^p \right)^{1/p} \leq 2^{\frac{1}{p}}h^{\frac{1}{p}-1} 
\]
and
\[
\frac{1}{2} h^{\frac{1}{p} - 1}\leq \left(\int_{\dot\Sigma}\left| \frac{d \beta_V}{d s} \right|^p \right)^{1/p} \leq 2^{\frac{1}{p}}h^{\frac{1}{p}-1}.
\]

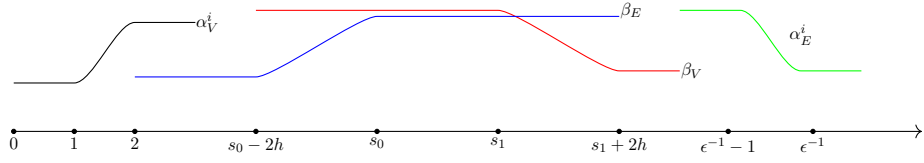
\begin{figure}
\centering
\def\h{2}
\def\hh{1}
\def\ss{0.3}
\def\hhh{0.1}
\def\hhhh{0.2}
\begin{tikzpicture}[scale = 0.8, every node/.style={scale=0.6}]
\draw[->] (0, 0) -- (15,0);
\filldraw[black] (0,0)  circle (1pt) node[anchor = north]{$0$};
\filldraw[black] (1,0)  circle (1pt) node[anchor = north]{$1$};
\filldraw[black] (2,0)  circle (1pt) node[anchor = north]{$2$};
\filldraw[black] (4,0)  circle (1pt) node[anchor = north]{$s_0-2h$};
\filldraw[black] (6,0)  circle (1pt) node[anchor = north]{$s_0$};
\filldraw[black] (8,0)  circle (1pt) node[anchor = north]{$s_1$};
\filldraw[black] (10,0)  circle (1pt) node[anchor = north]{$s_1+2h$};
\filldraw[black] (11.8,0)  circle (1pt) node[anchor = north]{$\epsilon^{-1}-1$};
\filldraw[black] (13.2,0)  circle (1pt) node[anchor = north]{$\epsilon^{-1}$};
\draw[green] (11,\h) -- (12, \h) .. controls (12 + \ss, \h) and (13 - \ss, \h - \hh) .. (13, \h - \hh) -- (14, \h - \hh);
\draw[red] (4,\h) -- (8, \h) .. controls (8 + \ss, \h) and (10 - \ss, \h - \hh) .. (10, \h - \hh) -- (11, \h - \hh);
\draw[blue] (2, \h -\hh -\hhh) -- (4, \h - \hh-\hhh) .. controls (4 + \ss, \h - \hh-\hhh) and (6 - \ss, \h -\hhh).. (6, \h-\hhh) -- (10, \h-\hhh);
\draw[] (0,\h - \hh - \hhhh ) -- (1, \h -\hh - \hhhh) .. controls (1 + \ss, \h - \hh - \hhhh) and (2 - \ss, \h  -\hhhh) .. (2, \h  -\hhhh) -- (3, \h  -\hhhh);
\node at (10.2, \h){$\beta_E$};
\node at (11.2, \h - \hh){$\beta_V$};
\node at (3.2,\h - \hhhh){$\alpha_V^i$};
\node at (13,\h  -\hhhh - \hhhh){$\alpha_E^i$};
\end{tikzpicture}
\caption{Cutoff functions on the cylindrical ends.}
\label{fig: cutoff functions}
\end{figure}

\subsection{Approximate Solution with Variations}

Let $S_i = [0,\infty) \times [0,1]$ denote the $i$-th end of $\dot\Sigma$, with coordinates
$(s,t)$.

By Assumptions~\ref{assumption: rigid Y tree} and~\ref{assumption: transversal intersection},
\[
T_{m_0}\mathscr{D}_{\cP{p}_1} \cap T_{m_0}\mathscr{D}_{\cP{p}_2} \cap T_{m_0}\mathscr{D}_{\cP{p}_3} = \{0\}.
\]

Let $\parameterSpace_V = \oplus_{i=1}^3 \parameterSpace_V^i \cong \R^{2n}$, with 
$$\parameterSpace_V^i = \{\alpha_V^i v : v \in T_{m_0}\mathscr{D}_{\cP{p}_i}\}\cong \R^{\ind \cP{p}_i},$$ where $\alpha_V^i$ is the bump
function on the $i$-th end of $\dot\Sigma$ such that it equals $1$ when $s \geq 2$ and it equals $0$ when $s \leq 1$  as  illustrated in
Figure~\ref{fig: cutoff functions}. Let $\parameterSpace_E = \oplus_{i=1}^3 \parameterSpace_E^i$, with 
$$\parameterSpace_E^i = \{\alpha_E^i v: v \in T_{m_0}^\perp \mathscr{D}_{\cP{p}_i}\} \cong \R^{n-\ind \cP{p}_i},$$
where $\alpha_E^i$ is the bump function on $S_i$ that equals $1$ for $s \leq \epsilon^{-1}-1$ and $0$ for $s \geq \epsilon^{-1}$. 
For any perturbation parameters
\[
(\td\xi, \xi_1, \xi_2, \xi_3, \tdlv, \constantVector) \in
W^{1,p}_\delta\bigl(\tdlm^* T(\C^n; \td{\tup\Lambda})\bigr)
\oplus \left( \oplus_{i=1}^3 W^{1,p}_{(\epsilon),0}(w_i^*TX) \right) \oplus \parameterSpace_V \oplus \parameterSpace_E
\]
as in Proposition~\ref{prop: add elements to fill cokernel}, we define the approximate map
$\app$ by interpolating between the local model and the edge models in a small neighborhood of
$(m_0, 0) \in X$, where the exponential map is given by addition:
\begin{enumerate}
\item On $S_i$ with $s > s_1 + 2h$:
\[
\app = w_i + \xi_i + \constantVector_i.
\]
\item On $\dot\Sigma \setminus \coprod_i S_i^{s \geq s_0 - 2h}$:
\[
\td{\app} = \tdlm + \td\xi + \tdlv.
\]
\item On $S_i$ with $s_0 - 2h < s < s_1 + 2h$:
\begin{align}
\td{\app}(s,t) =\,
&\beta_E(s)\bigl(\td{w}_i(s,t) - \td{L}_i(s,t) + \td\xi_i(s,t)  + \constantVector_i \bigr) \nonumber\\
&+ \beta_V(s)\bigl(\tdlm(s,t) - \td{L}_i(s,t) + \td\xi(s,t)  + \tdlv \bigr)
+ \td{L}_i(s,t), \label{eq:approx_sol}
\end{align}
\end{enumerate}
where $\td{L}_i$ is defined in Equation~\eqref{eqn: linear model}. It is straightforward to
verify that $\app$ is well-defined.

We check that $\app$ satisfies the Lagrangian boundary conditions. Since
$\td\Lambda_i = \R^n + \iu\tup{c}_i$, over the interpolation region we have
\[
\td{w}_i(s,0) \in \td\Lambda_i, \qquad \td{w}_i(s,1) \in \td\Lambda_{i-1},
\]
\[
\td{L}_i(s,0) \in \td\Lambda_i, \qquad \td{L}_i(s,1) \in \td\Lambda_{i-1},
\]
and
\begin{align*}
\td{\xi}_i(s,0) \in T_{\td{w}_i(s,0)}\td\Lambda_i \simeq \R^n, & \qquad
\td{\xi}_i(s,1) \in T_{\td{w}_i(s,1)}\td\Lambda_{i-1} \simeq \R^n,\\
\tdlv(s) \in \R^n, & \qquad \constantVector(s) \in \R^n.
\end{align*}
Therefore $\app(s,0) \in \td\Lambda_i$ and $\app(s,1) \in \td\Lambda_{i-1}$, as required.

We now compute $\overline\partial_J\app$. Write
\[
\overline\partial_J\app = \beta_V\Theta_V + \beta_E\Theta_E,
\]
where $\Theta_V$ is a section of $\wedge^{0,1}\dot\Sigma \otimes_\C \tdlm^*T\C^n$ and $\Theta_E$
is a section of $\coprod_i \wedge^{0,1}S_i \otimes_\C w_i^*TX$, so that both $\beta_V\Theta_V$
and $\beta_E\Theta_E$ are sections of $\wedge^{0,1}\dot\Sigma \otimes_\C \app^*TX$. Explicitly,
\[
\Theta_V
= \overline\partial\td\xi + \overline\partial\tdlv
+ \Bigl[\sum_{i=1}^3 \frac{\dd\beta_E}{\dd s}
  \bigl(\td{w}_i  - \td{L}_i + \td\xi_i +\constantVector_i \bigr) + \td{F}_0\Bigr]
  \otimes (\dd s + \iu\,\dd t),
\]
where $\td{F}_0$ is the error term arising from the difference between $J$ and $J_0$; under
Assumption~\ref{assumption: flat}, $\td{F}_0 = 0$. On the region $S_i$ with
$s \leq s_1+2h$,
\[
\Theta_E
= \Bigl[D_i\td\xi_i + D_i \constantVector_i
+ \frac{\dd\beta_V}{\dd s}\bigl(\tdlm - \td{L}_i + \td\xi  + \tdlv \bigr)
+ \td{F}_i\Bigr] \otimes (\dd s + \iu\,\dd t),
\]
while on $S_i$ with $s > s_1+2h$,
\[
\Theta_E = (D_i\xi_i + F_i) \otimes (\dd s + \iu\,\dd t).
\]

Here $D_i$ is defined in \ref{eqn: pushing to morse} (there $w = w_i$ and $D_i = D_w$.), which mediates between the linearized Cauchy-Riemann operator and the linearized gradient operator, and $F_i$ is the error term
\begin{equation}\label{eq:error term}
F_i = \overline\partial_J w_i + H_i + E_i,
\end{equation}
where $H_i = (\mathcal{L}_{w_i} - D_i)\xi_i$ is the difference between the actual
linearization $\mathcal{L}_{w_i}$ of $\overline\partial_J$ at $w_i$ and the operator $D_i$,
and $E_i$ is the type~1 quadratic remainder from Lemma~\ref{lemma: linearization}. By
Lemma~\ref{lemma: linearization}, $H_i$ satisfies the pointwise estimate
\begin{equation}\label{est:H}
|H_i| \leq C(\epsilon^2|\xi_i| + \epsilon|\nabla\xi_i|).
\end{equation}
Following the terminology of \cite{hutchings2009gluing}, a term $E(\xi)$ is \emph{type~1
quadratic} if
\[
E(\xi) = P(\xi) + Q(\xi) \cdot \nabla\xi,
\]
where $|P(\xi)| \leq C\epsilon|\xi|^2$ and $|Q(\xi)| \leq C|\xi|$ for some constant $C$. 

The $E(\xi)$ corresponds to the quadratic term $N(\xi)$ in the Floer-Picard Lemma \ref{lem:app-floer-picard} in Appendix. 

The pre-gluing equation $\overline\partial_J\app = 0$ is thus equivalent to the system
\begin{equation}
\Theta_V = 0, \qquad \Theta_E = 0.
\end{equation}

\section{Gluing} \label{section: gluing}
\subsection{Obstruction section}

Over the parameter space $\parameterSpace_V \times \parameterSpace_E$, we define a trivial obstruction bundle
\[
\mathcal O = \parameterSpace_V \times \parameterSpace_E \times \R^{2n} \times \R^n .
\]
Our goal is to construct an obstruction section of this bundle.

\subsubsection{Interior obstruction}
We first consider the obstruction arising from the interior model.  Set
\begin{equation}
\eta = \Theta_V - \overline{\p} \td\xi .
\end{equation}
Introduce the function spaces
\[
W = W^{1,p}_\delta(\tdlm^*(T\C^n; \td{\tup{\Lambda}})), 
\qquad 
W' = L^p_{\delta}\bigl(\wedge^{0,1}_{\dot\Sigma}\otimes_\C \tdlm^{*}T\C^{n} \bigr).
\]
The operator
\[
\overline{\p}|_{W}: W \to W'
\]
has a $2n$-dimensional cokernel.

For each puncture $p_i$ ($i=1,2,3$), define a subspace $Y_i \subset W'$ by
\[
Y_i = \op{span} \langle \frac{\dd \alpha_V^i}{\dd s} e_i^1, \dots, \frac{\dd \alpha_V^i}{\dd s} e_i^{\ind(\cP{p}_i)} \rangle,
\]
where 
\[  
\{e_i^1, e_i^2,\dots,e_i^{\ind \cP{p}_i}\}
\]
is an orthogonal basis of $T_{m_0}\mathscr{D}_{\cP{p}_i}$ such that $\|e_i^j\|_g = \epsilon$.  
These subspaces yield the direct sum decomposition
\[
W' = Y_1 \oplus Y_2 \oplus Y_3 \oplus \op{im}(\overline{\p}|_{W});
\]
and $\alpha_V^i$ is the cutoff function on $\dot\Sigma$ that equals $0$ for $s \leq 1$ and $1$ for $s \geq 2$ as illustrated in Figure~\ref{fig: cutoff functions}.
Denote by $\Pi_i: W' \to Y_i$ the projection onto $Y_i$ with respect to this splitting, and let $\Pi = \Pi_1+\Pi_2+\Pi_3$.

\begin{lemma}
The projections $\Pi, \Pi_1, \Pi_2, \Pi_3: W' \to W'$ are bounded.
\end{lemma}

Using $\eta$, the equation $\Theta_V = 0$ can be rewritten as
\[
\overline{\p}\td\xi + (1-\Pi)\eta + \Pi\eta = 0 .
\]

\subsubsection{Exterior obstruction}
Now consider the obstruction coming from the exterior model.  Set $\Sigma_i = [0,\infty) \times [0,1]$, and define
\[
V_i = W^{1,p}_{(\epsilon),0}(\Sigma_i, w_i^*TX),
\qquad
V_i' = L^{p}_{(\epsilon)}(\wedge^{0,1}_{\Sigma_i}\otimes_\C w_i^*TX).
\]
The restricted operator
\[
D_i|_{V_i}:V_i\to V_i'
\]
is not surjective.

Let $v_i^1, \dots, v_i^{n - \ind(\cP{p}_i)}$ be an orthogonal basis of the orthogonal complement $T_{m_0}^\perp \mathscr{D}_{\cP{p}_i}$ of $T_{m_0} \mathscr{D}_{\cP{p}_i}$ in $T_{m_0} M$ such that $\|v_i^j\|_g = \epsilon$.  Define
\[
Z_i=\op{span}\langle\frac{\dd \alpha_
V^i}{\dd s} v_i^1,\dots,\frac{\dd \alpha_E^i}{\dd s} v_i^{\,n-\ind(\cP{p}_i)}\rangle\subset V_i',
\]
where $\alpha_E^i$ is the cutoff function on $S_i$ that equals $1$ for $s \leq \epsilon^{-1} -1$ and $0$ for $s \geq \epsilon^{-1}$ as illustrated in Figure~\ref{fig: cutoff functions}.

\begin{lemma}
\[
\op{im}D_i|_{V_i} \oplus Z_i = V'_i .
\]
\end{lemma}
\begin{proof}
By Proposition~\ref{prop: fredholm index of edge}, we have
$\dim \op{coker}(D_i|_{V_i}) = \dim Z_i$. Hence it is enough to show that
$\op{im}D_i|_{V_i} + Z_i = V'_i$, or equivalently that no nonzero cokernel
element is orthogonal to $Z_i$.

Suppose there is a nonzero $b \in V_i'$ with
$\langle D_i a, b\rangle = 0$ for all $a \in V_i$ and
$\langle \dd\Phi(t)\rho, b\rangle = 0$ for all $\rho \in Z_i$. Then $b$
satisfies the adjoint equation \eqref{eqn: adjoint equation for edge} and the
same boundary conditions as in the proof of
Proposition~\ref{prop: fredholm index of edge}. In particular, $b$ is
$t$-independent and decays as $s \to +\infty$, so its value at $s=0$,
$b(0)$, is a nonzero vector in $T_{m_0}^\perp \mathscr{D}_{\cP{p}_i}$.

For any $v \in T_{m_0}^\perp \mathscr{D}_{\cP{p}_i}$, the element
$\rho_v = \frac{\dd \alpha_E^i}{\dd s} v$ belongs to $Z_i$. The pairing with
$b$ is then
\[
\langle b, \dd\Phi(t)\rho_v\rangle
= \int_0^\infty \left\langle b(s), \frac{\dd \alpha_E^i}{\dd s}(s) v \right\rangle \dd s
= \langle b(0), v \rangle,
\]
since $b$ is constant in $t$ and the cutoff integrates to $1$ near $s=0$.
Because $b(0)\neq 0$, one can choose $v$ so that this pairing is nonzero.
This contradicts the assumption that $b$ annihilates $Z_i$.

Therefore no such nonzero $b$ exists, and we conclude that
$\op{im}D_i|_{V_i} \oplus Z_i = V'_i$.
\end{proof}

Let $\Pi_i': V_i' \to Z_i$ denote the projection onto $Z_i$ with respect to this splitting.

Over each end $\Sigma_i$, define $\eta_i$ by
\[
\Theta_E = (D_i \xi_i + \eta_i)\otimes (\dd s + \iu \dd t).
\]

\subsubsection{The nonlinear system}
Our aim is to solve the following system:
\[
\begin{cases}
\overline{\p} \td\xi = (\Pi - 1) \eta, &\quad \Pi \eta = 0,\\[4pt]
D_i \xi_i = (\Pi_i' - 1) \eta_i, &\quad \Pi_i' \eta_i = 0 \quad (i=1,2,3).
\end{cases}
\]
Note that this system is nonlinear.

Define the space
\[
\mathbb W := W^{1,p}_{\delta}\bigl(\tdlm^{*}T(\C^{n};\td{\tup{\Lambda}})\bigr)
\bigoplus_{i=1}^3 W^{1,p}_{(\epsilon),0}(w_i^*TX).
\]

For each $i$, let 
\[
D_i^{-1}: \op{im}(D_i) \to W^{1,p}_{(\epsilon),0}(w_i^*TX)
\]
denote the inverse operator of
\[
D_i: W^{1,p}_{(\epsilon),0}(w_i^*TX) \to \op{im}D_i.
\]

Fix parameters $\tdlv\in\R^{2n}$ and $\constantVector\in\R^n$.  Define a map $\Psi: \mathbb W \to \mathbb W$ by
\begin{equation}\label{eq:Psi}
\Psi(\td{\xi}, \xi_1, \xi_2, \xi_3)
= -\bigl(\overline{\p}^{-1}(\Pi - 1)\eta,\;
D_1^{-1} (\Pi_1' - 1)\eta_1,\;
D_2^{-1}(\Pi_2' - 1) \eta_2,\;
D_3^{-1} (\Pi_3' - 1)\eta_3\bigr).
\end{equation}

The norm on $\mathbb W$ is given by
\begin{align}
\|(\td{\xi}, \xi_1,\xi_2,\xi_3)\|_{\W}
&= \| \td{\xi} \|_{W^{1,p}_\delta(\C^n)} + \frac{1}{\epsilon}\sum_{i=1}^3 \|\xi_i \|_{W^{1,p}_{(\epsilon)}(M)}, \label{normW}\\
\|\td{\xi}\|_{W^{1,p}_\delta(\C^n)}
&= \bigl\|\td{\xi} \, e^{\frac{\delta}{p}s}\bigr\|_{W^{1,p}(\C^n)}, \label{norminterior}\\
\|\xi_i\|_{W^{1,p}_{(\epsilon)}(M)}
&= \left(\int_\Sigma \bigl(\epsilon^2 |\xi_i|^p + \epsilon^{2-p} |\nabla \xi_i |^p\bigr) \right)^{\frac{1}{p}}, \label{normexterior}
\end{align}
where $\Sigma = [0, \infty) \times [0,1]$.  
For later use we also recall the $L^p_{(\epsilon)}$-norm
\[
\| \eta \|_{L^p_{(\epsilon)}(M)} = \left(\int_{\dot \Sigma} \epsilon^{2-p} |\eta|^{p} \right)^{1/p}.
\]

\textbf{Convention}: Throughout this section we use the same symbol for a section of a vector bundle and its representative in local coordinates. The two ``charts" --the chart of
$X$ and the chart of the tangent space $T_{(m_0,0)}X\cong\mathbb{C}^n$ --are related by a
blowup of factor $\eps$, so the norms differ accordingly.  We use
subscripts on the norm to indicate which chart is being used:
$\|\cdot\|_{L^p(M)}$ for the norm on $M$ and $\|\cdot\|_{L^p(\mathbb{C}^n)}$
for the norm on the tangent space.  The reader should bear in mind that
for the same geometric section $\xi_i$,
\[
  \|\xi_i\|_{L^p(\mathbb{C}^n)} \;=\; \frac{1}{\eps}\|\xi_i\|_{L^p(M)}.
\]
This is the reason for the factor $\frac{1}{\eps}$ before the terms $\|\xi_i \|_{W^{1,p}_{(\epsilon)}(M)}$ in the norm on $\mathbb W$. 

For any $\tdlv \in \parameterSpace_V$, we write $\tdlv = \tdlv_1 + \tdlv_2 + \tdlv_3$, where $\tdlv_i = \sum_{k=1}^{\ind(\cP{p}_i)} \tdlv_i^k \alpha_V^i e_i^k$. 
For any $r' > 0$, denote by $B_V(r')\subset \parameterSpace_V$ such that 
$$\left(\sum_{k=1}^{\ind(\cP{p}_i)}|\tdlv_i^k|^p \right)^{\frac{1}{p}} \leq r' \qquad i = 1, 2, 3.$$
For any $\constantVector \in \parameterSpace_E$, we write $\constantVector = (\constantVector_1, \constantVector_2, \constantVector_3)$, where $\constantVector_i = \sum_{k=1}^{n-\ind(\cP{p}_i)} \constantVector_i^k \alpha_E^i v_i^k$.
 For any $r' > 0$, denote by $B_E(r')\subset \parameterSpace_E$ such that $$\left(\sum_{k=1}^{n-\ind(\cP{p}_i)}|\constantVector_i^k|^p \right)^{\frac{1}{p}} \leq r' \qquad i = 1, 2, 3.$$

\begin{theorem}\label{thm:contraction}
There exists $\epsilon_0 > 0$ with the following property. For every $0 < \epsilon < \epsilon_0$, one can choose $r, r'>0$ so that, for every $\tdlv \in B_V(r')$ and $\constantVector\in B_E(r')$, the map $\Psi$ defined in \eqref{eq:Psi} is a contraction on the ball
\[
B = \{x \in \mathbb W : \|x\|_{\W} \leq r\}.
\]
More precisely,
\begin{enumerate}
    \item $\Psi(B) \subset B$;
    \item there exists a constant $\rho < 1$ such that $\|\Psi(x) - \Psi(x')\|_{\W} \leq \rho \|x - x'\|_{\W}$ for all $x, x' \in B$.
\end{enumerate}
\end{theorem}

\begin{proof}
Choose $r = K \epsilon^{\frac{1}{p}}$ and $r' = K\epsilon^{1-\frac{1}{p}}$, where $K > 0$ is a large constant independent of $\epsilon$ to be determined.  Denote by $C>0$ a generic constant independent of $\epsilon$.

\paragraph{(1).  Invariance of the ball.}
Write $(\td{\xi}', \xi_1', \xi_2', \xi_3') = \Psi(\td{\xi}, \xi_1, \xi_2, \xi_3)$.  We first estimate $\|\td{\xi}'\|_{W^{1,p}_\delta(\C^n)}$.

From the definition \eqref{eq:Psi} and the boundedness of $\overline{\p}^{-1}$,
\[
\|\td{\xi}'\|_{W^{1,p}_\delta(\C^n)}
= \| \overline{\p}^{-1}(\Pi - 1)\eta \|_{W^{1,p}_\delta(\C^n)}
\leq C \| (\Pi - 1)\eta \|_{L^p_\delta (\C^n)}.
\]

Since $(\Pi - 1)\overline{\p}\zeta = 0$, we have
\begin{align*}
\| (\Pi - 1)\eta \|_{L^p_\delta(\C^n)}
&= \Bigl\| (\Pi - 1)\Bigl[ \sum_{i=1}^3\frac{\dd\beta_E}{\dd s}\bigl(\td{w}_{i}+\td{\xi}_{i} - \td{L}_{i} + \constantVector_i \bigr)\Bigr]\otimes(\dd s + \iu \dd t) \Bigr\|_{L^p_\delta(\C^n)} \\
&\leq C \sum_{i=1}^3 \Bigl\| \frac{\dd\beta_E}{\dd s}\bigl(\td{w}_{i}+\td{\xi}_{i} - \td{L}_{i} + \constantVector_i \bigr) \Bigr\|_{L^p_\delta(\C^n)} \\
&\leq C \sum_{i=1}^3 \Bigl\| \frac{\dd\beta_E}{\dd s}(\td{w}_{i} - \td{L}_{i}) \Bigr\|_{L^p_\delta(\C^n)}
+ C \sum_{i=1}^3 \Bigl\| \frac{\dd\beta_E}{\dd s}\td{\xi}_{i} \Bigr\|_{L^p_\delta(\C^n)} \\
& + C \sum_{i=1}^3 \Bigl\| \frac{\dd\beta_E}{\dd s}\constantVector_i \Bigr\|_{L^p_\delta(\C^n)}.
\end{align*}

Recall that $\operatorname{supp}\frac{\dd \beta_E}{\dd s} \subset [s_0-2h, s_0]$.  For the first term,
\begin{align}
\Bigl\| \frac{\dd\beta_E}{\dd s}(\td{w}_{i} - \td{L}_{i}) \Bigr\|_{L^p_\delta(\C^n)}
&\leq e^{\frac{\delta s_0}{p}}\Bigl\| \frac{\dd\beta_E}{\dd s}(\td{w}_{i} - \td{L}_{i}) \Bigr\|_{L^p(\C^n)} \notag\\
&\leq C e^{\frac{\delta s_0}{p}} \epsilon s_0^2 \notag\\
&\leq C \epsilon^{\frac{2}{p} + \sigma} \leq \frac{1}{100} r, \label{est:wL}
\end{align}
where the second inequality follows from Lemma~\ref{lemma: asymptotics of w_i} ($\|\td{w}_i - \td{L}_i\|_{\C^n} \leq C\epsilon s^2$ on the support), the third uses $e^{\frac{\delta s_0}{p}} = \epsilon^{-(1-\frac{2}{p}-\sigma)}$ and absorbs the $s_0^2 = O((\ln\epsilon^{-1})^2)$ factor into the exponent.

For the second term, we claim:
\begin{claim}
\begin{align}
\Bigl\| \frac{\dd\beta_E}{\dd s}\td{\xi}_{i} \Bigr\|_{L^p_\delta(\C^n)}
\leq C \epsilon^{\frac{\sigma}{2}} \epsilon^{-1}\|\xi_i\|_{W^{1,p}_{(\epsilon)}(M)}
\end{align}
\end{claim}
\begin{proof}[Proof of Claim]
Let $q$ be the conjugate exponent, $\frac1p+\frac1q=1$, and put
\[
 a_i(s,t)=\dd\Phi_i(t)^{-1}\xi_i(s,t).
\]
By the normalization in $W^{1,p}_{(\epsilon),0}$,
\[
\int_0^1 a_i(0,t)\,\dd t=0.
\]
Apply the anchored thin-ribbon estimate, Proposition~\ref{prop:app-anchored-thin-ribbon}, to $a_i$ on the strip
$R_{s_0}=[0,s_0]\times[0,1]$.  The hypotheses are satisfied componentwise by the
zero-average condition at $s=0$.  Hence, for
$(s,t)\in[s_0-2h,s_0]\times[0,1]$,
\begin{align}
|a_i(s,t)|
&\leq C(s_0+1)^{1/q}\epsilon^{1-\frac2p}
\|a_i\|_{W^{1,p}_{(\epsilon)}(R_{s_0})} \notag\\
&\leq C(s_0+1)^{1/q}\epsilon^{1-\frac2p}
\|\xi_i\|_{W^{1,p}_{(\epsilon)}(M)}. \label{eq:xi-pointwise-thin}
\end{align}
Here the last inequality uses the definition of the $W^{1,p}_{(\epsilon)}$ norm and the uniform bounds on $\dd\Phi_i(t)$ and $\dd\Phi_i(t)^{-1}$.  The same pointwise bound holds for $\xi_i$ after increasing $C$.

Using \eqref{eq:xi-pointwise-thin}, the support condition
$\operatorname{supp}\frac{\dd\beta_E}{\dd s}\subset[s_0-2h,s_0]$, and
$\|\frac{\dd\beta_E}{\dd s}\|_{L^p([s_0-2h,s_0])}\leq C h^{-1/q}$, H\"older's inequality gives
\begin{align}
\Bigl\| \frac{\dd\beta_E}{\dd s}\td{\xi}_{i} \Bigr\|_{L^p_\delta(\C^n)}
&\leq e^{\frac{\delta}{p}s_0}
\Bigl\| \frac{\dd\beta_E}{\dd s}\td{\xi}_{i} \Bigr\|_{L^p(\C^n)} \notag\\
&\leq C\epsilon^{-1}e^{\frac{\delta}{p}s_0}h^{-1/q}
\|\xi_i\|_{L^\infty([s_0-2h,s_0]\times[0,1],M)} \notag\\
&\leq C h^{-1/q}(s_0+1)^{1/q}\epsilon^{\sigma}
\bigl(\epsilon^{-1}\|\xi_i\|_{W^{1,p}_{(\epsilon)}(M)}\bigr). \label{est:xii}
\end{align}
In the second line we used the rescaled-chart relation $|\td{\xi}_i|_{\C^n}=\epsilon^{-1}|\xi_i|_M$, and in the last line we used
$e^{\frac{\delta s_0}{p}}=\epsilon^{-(1-\frac2p-\sigma)}$.
Finally, $h=s_0/4$, so $h^{-1/q}(s_0+1)^{1/q}$ is uniformly bounded for
$s_0\geq1$.  After decreasing $\epsilon_0$ if necessary,
\[
 h^{-1/q}(s_0+1)^{1/q}\epsilon^{\sigma}
 \leq C\epsilon^{\frac\sigma2},
\]
which proves the claim.
\end{proof}

For the third term,
\begin{align}
\Bigl\| \frac{\dd\beta_E}{\dd s}\constantVector_i \Bigr\|_{L^p_\delta(\C^n)}
&\leq \Bigl\| e^{\frac{\delta}{p} s} \frac{\dd\beta_E}{\dd s}\constantVector_i \Bigr\|_{L^p(\C^n)} \notag\\
&\leq e^{\frac{\delta}{p} s_0} \Bigl\| \frac{\dd\beta_E}{\dd s}\constantVector_i \Bigr\|_{L^p(\C^n)} \notag\\
&\leq C e^{\frac{\delta}{p} s_0} h^{-\frac{1}{q}} \|\constantVector_i\| \leq C \epsilon^{-1 + \frac{2}{p} + \sigma} h^{-\frac{1}{q}} r' \leq \frac{1}{100} r. \label{est:constantVector}
\end{align}

Combining \eqref{est:wL}, \eqref{est:xii}, and \eqref{est:constantVector}, and using
$\epsilon^{-1}\|\xi_i\|_{W^{1,p}_{(\epsilon)}(M)}\leq r$ on the ball,
\[
\|\td{\xi}'\|_{W^{1,p}_\delta(\C^n)}
\leq C \epsilon^{\frac{2}{p} + \sigma}
+ C\epsilon^{\frac\sigma2} r
+ C\epsilon^{-1+\frac{2}{p}+\sigma}r'
\leq \frac{3}{100} r.
\]

Now we estimate $\frac{1}{\epsilon}\|\xi_i'\|_{W^{1,p}_{(\epsilon)}(M)}$. 
\begin{align*}
\frac{1}{\epsilon}\|\xi_i'\|_{W^{1,p}_{(\epsilon)}(M)}
&= \frac{1}{\epsilon} \| D_i^{-1} (\Pi_i' - 1) \eta_i \|_{W^{1,p}_{(\epsilon)}(M)} \\
& \leq \frac{1}{\epsilon} C \| (\Pi_i' - 1) \eta_i \|_{L^p_{(\epsilon)}(M)} \\
&\leq \frac{1}{\epsilon} C \Bigl\| (\Pi_i' - 1) \Bigl( \frac{\dd\beta_{V}}{\dd s}\bigl(\lm+{\xi}-{L}_{i} + \lv  \bigr) + F_{i}\Bigr) \Bigr\|_{L^p_{(\epsilon)}(M)} \\
&\leq \frac{1}{\epsilon} C  \epsilon^{\frac{2}{p} - 1} \Bigl\| \frac{\dd\beta_{V}}{\dd s}\bigl(\lm + {\xi}- {L}_{i} + \lv  \bigr) + F_{i} \Bigr\|_{L^p(M)} \\
&\leq  C \epsilon^{\frac{2}{p} - 1}  \Bigl\| \frac{\dd\beta_{V}}{\dd s}\bigl(\tdlm+ \td{\xi}-\td{L}_{i} + \tdlv \bigr) \Bigr\|_{L^p(\C^n)} + C \epsilon^{\frac{2}{p} - 2} \| F_{i} \|_{L^p(M)}.
\end{align*}

We treat each part separately.  First,
\begin{align}
 C\epsilon^{\frac{2}{p}-1} \Bigl\| \frac{\dd\beta_{V}}{\dd s} \td{\xi} \Bigr\|_{L^p(\C^n)}
&\leq C\epsilon^{\frac{2}{p}-1} e^{-\frac{\delta}{p}s_1} \|\td{\xi}\|_{W^{1,p}_\delta(\C^n)} \notag \\
& \leq C \epsilon^{\sigma} \|\td{\xi}\|_{W^{1,p}_\delta(\C^n)}, \label{est:xii2}
\end{align}
where we used $e^{-\frac{\delta}{p}s_1}\epsilon^{\frac{2}{p}-1} \leq \epsilon^\sigma$ by the choice of $s_1$.  Since $C\epsilon^\sigma < \frac{1}{2}$ for small $\epsilon$, this term can be absorbed into the left-hand side.
Similarly,
\begin{align}
C \epsilon^{\frac{2}{p} - 1} \Bigl\| \frac{\dd\beta_{V}}{\dd s} \tdlv \Bigr\|_{L^p(\C^n)} &\leq C  \epsilon^{\frac{2}{p} - 1}  h^{\frac{1}{p}-1}\bigl(\sum_{i,j}|\varsigma_i^j|^p\bigr)^{\frac{1}{p}} \leq C \epsilon^{\frac{2}{p} - 1}  h^{\frac{1}{p}-1} r'  \leq \frac{1}{100}  r, \label{est:zeta}\\
C \epsilon^{\frac{2}{p} - 1} \Bigl\| \frac{\dd\beta_{V}}{\dd s} (\tdlm - \td{L}_{i}) \Bigr\|_{L^p(\C^n)} &\leq C \epsilon^{\frac{2}{p} - 1}  h^{\frac{1}{p}-1} e^{-\pi s_1}
= C \epsilon^{\frac{2}{p} - 1 + \frac{\pi p}{\delta}( 1 - \frac{2}{p} + \sigma)}  h^{\frac{1}{p}-1}
\leq \frac{1}{100} r, \label{est:L}
\end{align}
by choosing $\epsilon$ sufficiently small. 
To get \eqref{est:L}, we claim
\[
\frac{2}{p} - 1 + \frac{\pi p}{\delta}\Bigl(1 - \frac{2}{p} + \sigma\Bigr) > 0.
\]
But this can be guaranteed by choosing $\delta$ sufficiently small.

For the nonlinear term $F_i$, from \eqref{eq:error term} we write $F_i = \overline{\p}_J w_i + H_i + P_i(\xi_i) + Q_i(\xi_i)\cdot\nabla\xi_i$.  For the first term, from Lemma~\ref{lemma: dbar J estimate for edge model},
\begin{align}
\epsilon^{\frac{2}{p} - 2} \|\overline{\p}_J w_i\|_{L^p(M)}
& \leq C \epsilon^{\frac{1}{p}} \leq \frac{C}{K} r \leq \frac{1}{100}  r. \label{est:dbar}
\end{align}
where we choose $K$ sufficiently large.  For the $H_i$ term, using \eqref{est:H} and the norm estimates $\|\xi_i\|_{L^p(M)} \leq \epsilon^{-2/p}\|\xi_i\|_{W^{1,p}_{(\epsilon)}(M)}$ and $\|\nabla\xi_i\|_{L^p(M)} \leq \epsilon^{1-2/p}\|\xi_i\|_{W^{1,p}_{(\epsilon)}(M)}$,
\begin{align}
\epsilon^{\frac{2}{p}-2}\|H_i\|_{L^p(M)}
&\leq C\epsilon^{\frac{2}{p}-2}\bigl(\epsilon^2\|\xi_i\|_{L^p(M)} + \epsilon\|\nabla\xi_i\|_{L^p(M)}\bigr) \notag\\
&\leq C\epsilon^{\frac{2}{p}-2}\bigl(\epsilon^2\cdot\epsilon^{-\frac{2}{p}} + \epsilon\cdot\epsilon^{1-\frac{2}{p}}\bigr)\|\xi_i\|_{W^{1,p}_{(\epsilon)}(M)} \notag\\
&= C\|\xi_i\|_{W^{1,p}_{(\epsilon)}(M)} \leq C\epsilon r \leq \frac{1}{100} r, \label{est:H-ball}
\end{align}
for $\epsilon$ sufficiently small, using $\|\xi_i\|_{W^{1,p}_{(\epsilon)}(M)} \leq \epsilon r$.  For the other two terms in $F_i$, we need to use the following Sobolev embedding, which is proved in the appendix.

\begin{proposition}[Sobolev embedding for degenerating domain]\label{prop:sob}
For any $\xi_i \in W^{1,p}_{(\epsilon)}(M)$,
\[
\|\xi_i\|_{C^0(M)} \leq C \epsilon^{-\frac{1}{p}} \|\xi_i\|_{W^{1,p}_{(\epsilon)}(M)}.
\]
\end{proposition}

Using Proposition~\ref{prop:sob},
\begin{align}
\epsilon^{\frac{2}{p} - 2} \| P_i(\xi_i) \|_{L^p(M)}
&\leq  C \epsilon^{\frac{2}{p} - 1} \Bigl(\int_{\Sigma_i} (|\xi_i|^2)^p\Bigr)^{\frac{1}{p}} \notag\\
&\leq  C \epsilon^{\frac{2}{p} - 1} \|\xi_i\|_{C^0(M)} \|\xi_i\|_{L^p(M)} \notag\\
&\leq  C \epsilon^{\frac{1}{p} - 1} \|\xi_i\|_{W^{1,p}_{(\epsilon)}(M)} \|\xi_i\|_{L^p(M)} \notag\\
&\leq   C \epsilon^{-\frac{1}{p} - 1} \|\xi_i\|_{W^{1,p}_{(\epsilon)}(M)} \cdot  \|\xi_i\|_{W^{1,p}_{(\epsilon)}(M)} \notag\\
&\leq  C \epsilon^{1-\frac{1}{p}} r^2  \leq  \frac{1}{100} r, \label{est:P}
\end{align}
For the part,
\begin{align}
\epsilon^{\frac{2}{p} - 2} \| Q_i(\xi_i)\cdot\nabla\xi_i\|_{L^p(M)}
&\leq C \epsilon^{\frac{2}{p} - 2} \|\xi_i\|_{C^0(M)} \|\nabla\xi_i\|_{L^p(M)} \notag\\
&\leq C \epsilon^{\frac{1}{p} - 2} \|\xi_i\|_{W^{1,p}_{(\epsilon)}(M)} \|\nabla\xi_i\|_{L^p(M)} \notag\\
&\leq C \epsilon^{-\frac{1}{p} - 1} \|\xi_i\|_{W^{1,p}_{(\epsilon)}(M)} \cdot \|\xi_i\|_{W^{1,p}_{(\epsilon)}(M)} \notag\\
&\leq C \epsilon^{1-\frac{1}{p}} r^2 \leq \frac{1}{100} r. \label{est:Q}
\end{align}

Collecting the estimates \eqref{est:xii2}--\eqref{est:Q} we obtain
\[
\epsilon^{-1}\|\xi_i'\|_{W^{1,p}_{(\epsilon)}(M)} \leq \frac{1}{10}r.
\]
Together with the bound for $\|\td{\xi}'\|$, this gives $\|\Psi(\td{\xi},\xi_1,\xi_2,\xi_3)\|_{\W} \leq r$, i.e. $\Psi(B)\subset B$.

\paragraph{(2).  Contraction property.}
Take two elements $(\td{\xi},\xi_1,\xi_2,\xi_3)$ and $(\underline{\td{\xi}},\underline{\xi}_1,\underline{\xi}_2,\underline{\xi}_3)$ in $\mathbb W$, and denote their images under $\Psi$ by $(\td{\xi}',\xi_1',\xi_2',\xi_3')$ and $(\underline{\td{\xi}}',\underline{\xi}_1',\underline{\xi}_2',\underline{\xi}_3')$.  We estimate the difference.

For the interior component,
\begin{align*}
\|\td{\xi}' - \underline{\td{\xi}}' \|_{W^{1,p}_\delta(\C^n)}
&\leq C \| (\Pi-1)(\eta-\underline{\eta}) \|_{L^p_\delta(\C^n)} \\
&\leq C \sum_{i=1}^3 \Bigl\| \frac{\dd\beta_E}{\dd s}(\td{\xi}_{i} - \underline{\td{\xi}}_{i}) \Bigr\|_{L^p_\delta(\C^n)} \\
&\leq C \sum_{i=1}^3 e^{\frac{\delta}{p}s_0}\Bigl\| \frac{\dd\beta_E}{\dd s}(\td{\xi}_{i} - \underline{\td{\xi}}_{i}) \Bigr\|_{L^p(\C^n)} \\
&\leq C \sum_{i=1}^3 e^{\frac{\delta}{p}s_0} \epsilon^{-1} \Bigl\| \frac{\dd\beta_E}{\dd s}({\xi}_{i} - \underline{{\xi}}_{i}) \Bigr\|_{L^p(M)} \\
&\leq C \sum_{i=1}^3 \epsilon^{\frac{\sigma}{2}} \epsilon^{-1}  \|{\xi}_{i} - \underline{{\xi}}_{i}\|_{W^{1,p}_{(\epsilon)}(M)} \\
&\leq \frac{1}{100} \epsilon^{-1} \sum_{i=1}^3 \|{\xi}_{i} - \underline{{\xi}}_{i}\|_{W^{1,p}_{(\epsilon)}(M)}.
\end{align*}

For the exterior components,
\begin{align*}
\epsilon^{-1}\|\xi_i' - \underline{\xi}_i'\|_{W^{1,p}_{(\epsilon)}(M)}
&\leq C \epsilon^{-1} \| (\Pi_i'-1)(\eta_i-\underline{\eta}_i) \|_{L^p_{(\epsilon)}(M)} \\
&\leq C \epsilon^{\frac{2}{p}-2}\| (\Pi_i'-1)(\eta_i-\underline{\eta}_i) \|_{L^p(M)} \\
&\leq C \epsilon^{\frac{2}{p}-2} \Bigl\| \frac{\dd\beta_V}{\dd s}({\xi} - \underline{{\xi}}) \Bigr\|_{L^p(M)}
+ C \epsilon^{\frac{2}{p}-2} \|F_i-\underline{F}_i\|_{L^p(M)}.
\end{align*}

The first term is controlled by
\begin{align*}
\epsilon^{\frac{2}{p}-2} \Bigl\| \frac{\dd\beta_V}{\dd s}({\xi} - \underline{{\xi}}) \Bigr\|_{L^p(M)}
&\leq \epsilon^{\frac{2}{p}-1} \Bigl\| \frac{\dd\beta_V}{\dd s}(\td{\xi} - \underline{\td{\xi}}) \Bigr\|_{L^p(\C^n)} \\
&\leq \epsilon^{\frac{2}{p}-1} e^{-\frac{\delta}{p}s_1} \Bigl\| \frac{\dd\beta_V}{\dd s}(\td{\xi} - \underline{\td{\xi}}) \Bigr\|_{L^p_\delta(\C^n)} \\
&\leq \epsilon^{\sigma} \|\td{\xi} - \underline{\td{\xi}}\|_{W^{1,p}_\delta(\C^n)}.
\end{align*}

For the second term, we have 
\begin{align*}
  \epsilon^{\frac{2}{p}-2} \|F_i-\underline{F}_i\|_{L^p(M)} 
 & \leq \epsilon^{\frac{2}{p}-2}\|H_i(\xi_i) - H_i(\underline{\xi}_i)\|_{L^p(M)} \\
 & + \epsilon^{\frac{2}{p}-2}  \| P_i(\xi_i) - P_i(\underline{\xi}_i)\|_{L^p(M)} \\
 & + \epsilon^{\frac{2}{p}-2} \| Q_i(\xi_i)\cdot\nabla\xi_i - Q_i(\underline{\xi}_i)\cdot\nabla\underline{\xi}_i \|_{L^p(M)}.
\end{align*}

For the $H_i$ difference, since $H_i$ is linear in $\xi_i$, $H_i(\xi_i) - H_i(\underline{\xi}_i) = H_i(\xi_i - \underline{\xi}_i)$, and by \eqref{est:H},
\begin{align*}
\epsilon^{\frac{2}{p}-2}\|H_i(\xi_i-\underline{\xi}_i)\|_{L^p(M)}
&\leq C\epsilon^{\frac{2}{p}-2}\bigl(\epsilon^2\cdot\epsilon^{-\frac{2}{p}} + \epsilon\cdot\epsilon^{1-\frac{2}{p}}\bigr)\|\xi_i-\underline{\xi}_i\|_{W^{1,p}_{(\epsilon)}(M)}\\
&= C\|\xi_i-\underline{\xi}_i\|_{W^{1,p}_{(\epsilon)}(M)}
\leq C\epsilon\cdot\bigl(\epsilon^{-1}\|\xi_i-\underline{\xi}_i\|_{W^{1,p}_{(\epsilon)}(M)}\bigr)\\
&\leq \frac{1}{100}\bigl(\epsilon^{-1}\|\xi_i-\underline{\xi}_i\|_{W^{1,p}_{(\epsilon)}(M)}\bigr),
\end{align*}
for $\epsilon$ sufficiently small.  We use again Proposition~\ref{prop:sob} for the remaining terms:
\begin{align*}
\epsilon^{\frac{2}{p}-2}  \| P_i(\xi_i) - P_i(\underline{\xi}_i)\|_{L^p(M)}
&\leq C \epsilon^{\frac{2}{p}-1} \Bigl(\int_{\Sigma_i} (|\xi_i|+|\underline{\xi}_i|)^{p} |\xi_i-\underline{\xi}_i|^p\Bigr)^{\frac{1}{p}} \\
&\leq C \epsilon^{\frac{2}{p}-1} (\|\xi_i\|_{C^0(M)} + \|\underline{\xi}_i\|_{C^0(M)}) \Bigl(\int_{\Sigma_i} |\xi_i-\underline{\xi}_i|^p\Bigr)^{\frac{1}{p}} \\
&\leq C \epsilon^{\frac{1}{p}-1} (\|\xi_i\|_{W^{1,p}_{(\epsilon)}(M)} + \|\underline{\xi}_i\|_{W^{1,p}_{(\epsilon)}(M)}) \|\xi_i-\underline{\xi}_i\|_{L^p(M)} \\
&\leq C \epsilon^{\frac{1}{p}-1} (\|\xi_i\|_{W^{1,p}_{(\epsilon)}(M)} + \|\underline{\xi}_i\|_{W^{1,p}_{(\epsilon)}(M)}) \cdot \epsilon^{-\frac{2}{p}} \|\xi_i-\underline{\xi}_i\|_{W^{1,p}_{(\epsilon)}(M)} \\
&\leq C \epsilon^{1-\frac{1}{p}} r \cdot \bigl(\epsilon^{-1}\|\xi_i-\underline{\xi}_i\|_{W^{1,p}_{(\epsilon)}(M)}\bigr) \\
&\leq \frac{1}{100} \bigl(\epsilon^{-1}\|\xi_i-\underline{\xi}_i\|_{W^{1,p}_{(\epsilon)}(M)}\bigr).
\end{align*}
Finally,
\begin{align*}
& \epsilon^{\frac{2}{p}-2} \| Q_i(\xi_i)\cdot\nabla\xi_i - Q_i(\underline{\xi}_i)\cdot\nabla\underline{\xi}_i \|_{L^p(M)} \\
 \leq &  C \epsilon^{\frac{2}{p}-2} \| \xi_i \cdot (\nabla\xi_i - \nabla\underline{\xi}_i) \|_{L^p(M)} + C \epsilon^{\frac{2}{p}-2} \| (\xi_i-\underline{\xi}_i) \cdot \nabla\underline{\xi}_i\|_{L^p(M)} \\
\leq & C \epsilon^{\frac{1}{p}-2} \|\xi_i\|_{W^{1,p}_{(\epsilon)}(M)} \|\nabla\xi_i - \nabla\underline{\xi}_i\|_{L^p(M)} + C \epsilon^{\frac{1}{p}-2} \|\xi_i-\underline{\xi}_i\|_{W^{1,p}_{(\epsilon)}(M)} \|\nabla\underline{\xi}_i\|_{L^p(M)} \\
\leq & C \epsilon^{-\frac{1}{p}-1} \|\xi_i\|_{W^{1,p}_{(\epsilon)}(M)} \|\xi_i-\underline{\xi}_i\|_{W^{1,p}_{(\epsilon)}(M)}  + C \epsilon^{-\frac{1}{p}-1} \|\xi_i-\underline{\xi}_i\|_{W^{1,p}_{(\epsilon)}(M)} \cdot \|\underline{\xi}_i\|_{W^{1,p}_{(\epsilon)}(M)} \\ 
\leq & C \epsilon^{1-\frac{1}{p}} \epsilon^{-1} \|\xi_i\|_{W^{1,p}_{(\epsilon)}(M)} \epsilon^{-1} \|\xi_i-\underline{\xi}_i\|_{W^{1,p}_{(\epsilon)}(M)}  \\ 
\leq & \frac{1}{100} \bigl(\epsilon^{-1}\|\xi_i-\underline{\xi}_i\|_{W^{1,p}_{(\epsilon)}(M)}\bigr).
\end{align*}

Putting these estimates together shows that $\Psi$ is a contraction.  This completes the proof of Theorem~\ref{thm:contraction}.
\end{proof}

By the contraction mapping theorem, for every admissible $(\tdlv,\constantVector)$ there exists a unique fixed point $(\td\xi,\xi_1,\xi_2,\xi_3)$ of $\Psi$.  Substituting this fixed point into $\eta$ and $\eta_i$ makes them functions of $\tdlv$ and $\constantVector$ only.

It remains to choose $\tdlv$ and $\constantVector$ so that the projection conditions $\Pi\eta = 0$ and $\Pi_i'\eta_i = 0$ hold.  We therefore define the obstruction section
\[
\obstructionSection: B_V(r') \times B_E(r') \longrightarrow Y_1 \times Y_2 \times Y_3 \times Z_1 \times Z_2 \times Z_3 \simeq \R^{2n} \times \R^n
\]
by
\[
\obstructionSection(\tdlv, \constantVector) = (\Pi_1(\eta),\; \Pi_2(\eta),\; \Pi_3(\eta),\; \Pi_1'(\eta_1),\; \Pi_2'(\eta_2),\; \Pi_3'(\eta_3)).
\]

\subsection{The zeroes of \texorpdfstring{$\mathfrak s$}{s}}
In this section, we show that $\mathfrak s$ has one zero counted with signs. To be precise, we have:
\begin{theorem}\label{thm: degree one}
The obstruction section $\obstructionSection$ has no zero on $\p (B_V(r') \times B_E(r'))$. The normalized boundary map $\p (B_V(r') \times B_E(r')) \to S^{3n-1} \subset \R^{3n}$ given by
$$v \mapsto \frac{\obstructionSection(v)}{\|\obstructionSection(v)\|}$$
has degree one.
\end{theorem}

We first consider the so-called linearized obstruction section $\obstructionSection_0$ defined by

\begin{align*}
  \obstructionSection_0(\tdlv, \constantVector) = & \left(\Pi_1(\overline{\p}\tdlv), \Pi_2(\overline{\p}\tdlv), \Pi_3(\overline{\p}\tdlv), \Pi_1'(D_1 \constantVector_1), \Pi_2'(D_2 \constantVector_2), \Pi_3'(D_3 \constantVector_3) \right) \\ 
  = & \left(\overline{\p}\tdlv_1,\overline{\p}\tdlv_2,\overline{\p}\tdlv_3, D_1 \constantVector_1, D_2 \constantVector_2, D_3 \constantVector_3 \right).
\end{align*}

With respect to the bases
\[
\{\alpha_V^i e_i^j \mid i=1,2,3,\ j=1,\dots,\ind(\cP{p}_i)\}
\]
for $B_V(r')$,
\[
\{\alpha_E^i v_i^j \mid i=1,2,3,\ j=1,\dots,n-\ind(\cP{p}_i)\}
\]
for $B_E(r')$,
\[
\bigl\{\tfrac{\dd\alpha_V^i}{\dd s} e_i^j \mid j=1,\dots,\ind(\cP{p}_i)\bigr\}
\]
for $Y_i$, and
\[
\bigl\{\tfrac{\dd\alpha_E^i}{\dd s} v_i^j \mid j=1,\dots,n-\ind(\cP{p}_i)\bigr\}
\]
for $Z_i$,
the map $\obstructionSection_0$ is represented by the identity map. In particular,
$\obstructionSection_0$ has degree one.

Now we show that $\obstructionSection$ is homotopic to $\obstructionSection_0$ through nonvanishing sections on $$\p (B_V(r') \times B_E(r')).$$ 
Define the homotopy $$\obstructionSection_\tau(\tdlv, \constantVector) = (\Pi_1(\eta_V^\tau), \Pi_2(\eta_V^\tau), \Pi_3(\eta_V^\tau), \Pi_1'(\eta_E^{1,\tau}), \Pi_2'(\eta_E^{2,\tau}), \Pi_3'(\eta_E^{3,\tau})),$$ 
where $$\eta_V^\tau =  \overline{\p}\td{\lv}
+(1-\tau)\Bigl[
\sum_{i=1}^3\frac{\dd\beta_E}{\dd s}\bigl(\td{w}_{i}+\td{\xi}_{i}  - \td{L}_{i} + \constantVector_i \bigr)
+F_{0}\Bigr]\otimes(\dd s + \iu \dd t)$$ 
and
$$\eta_E^{i,\tau} = 
\left[ D_i \constantVector_i + (1 - \tau) \left(\frac{\dd\beta_{V}}{\dd s}\bigl(\tdlm+\td{\xi}-\td{L}_{i} + \tdlv \bigr)
+F_{i} \right) \right]\otimes(\dd s + \iu \dd t),$$
and $(\td{\xi}, \xi_1, \xi_2, \xi_3)$ is the fixed point of $\Psi$ for given $\tdlv$ and $\constantVector$.

\begin{proposition}\label{prop: homotopy}
$\obstructionSection_\tau$ is nonvanishing on $\p (B_V(r') \times B_E(r'))$ for all $\tau \in [0,1]$.
\end{proposition}

\begin{proof}
Note that $$\p (B_V(r') \times B_E(r')) = [\p(B_V(r')) \times B_E(r')] \cup [B_V(r') \times \p(B_E(r'))].$$
To show $\obstructionSection_\tau \neq 0$ we show that over $\p(B_V(r')) \times B_E(r')$, the component $\Pi_j(\eta_V^\tau) \neq 0$ for some $j$; over $B_V(r') \times \p(B_E(r'))$, the component $\Pi_i'(\eta_E^{i,\tau}) \neq 0$ for some $i$.


\subsection*{Boundary 1: Over $\p(B_V(r')) \times B_E(r')$}
Note that 
\begin{align*}
\p(B_V(r')) &= \left(\p(B_V^1(r')) \times B_V^2(r') \times B_V^3(r')\right) \\
&\quad\cup \left(B_V^1(r') \times \p(B_V^2(r')) \times B_V^3(r')\right) \\
&\quad\cup \left(B_V^1(r') \times B_V^2(r') \times \p(B_V^3(r'))\right),
\end{align*}
where 
\[
B_V^i(r') =
\left\{
\tdlv_i = \sum_{j=1}^{\ind(\cP{p}_i)} \tdlv_i^j\,\alpha_V^i e_i^j
  \in \parameterSpace_V^i \,\bigg|\,
  \biggl(\sum_{j=1}^{\ind(\cP{p}_i)} |\tdlv_i^j|^p\biggr)^{1/p} = r'
\right\}.
\]
For any $(\tdlv, \constantVector) \in \p(B_V(r')) \times B_E(r')$, we have 
\[
\left(\sum_{k=1}^{\ind(\cP{p}_j)} |\tdlv_j^k|^p \right)^{\frac{1}{p}} = r', \qquad \text{for some } j \in \{1, 2, 3\}.
\]
We estimate $\Pi_j(\eta_V^\tau)$ in $L^p(\C^n)$ norm.

From the homotopy definition,
\[
\eta_V^\tau = \overline{\p}\td{\lv}
+ (1-\tau)\Bigl[\sum_{i=1}^3\frac{\dd\beta_E}{\dd s}\bigl(\td{w}_{i}+\xi_{i} - \td{L}_{i} + \constantVector_i \bigr)\Bigr]\otimes(\dd s + \iu\,\dd t),
\]
so by the triangle inequality,
\begin{align}
\|\Pi_j(\eta_V^\tau)\|_{L^p(\C^n)}
&\geq \|\Pi_j(\overline{\p}\td{\lv})\|_{L^p(\C^n)}
- C\sum_{i=1}^3\Bigl\|\Pi_j\Bigl(\frac{\dd\beta_E}{\dd s}(\td{w}_i - \td{L}_i)\Bigr)\Bigr\|_{L^p(\C^n)} \notag\\
&\quad - C\sum_{i=1}^3\Bigl\|\Pi_j\Bigl(\frac{\dd\beta_E}{\dd s}\xi_i\Bigr)\Bigr\|_{L^p(\C^n)}
- C\sum_{i=1}^3\Bigl\|\Pi_j\Bigl(\frac{\dd\beta_E}{\dd s}\constantVector_i\Bigr)\Bigr\|_{L^p(\C^n)}.
\label{eqn: emmy}
\end{align}

\medskip
\noindent\textbf{Main term.}
 $$\|\Pi_j(\overline{\p}\td{\lv})\|_{L^p(\C^n)} = \|\overline{\p}\td{\lv}_j\|_{L^p(\C^n)} \geq \left(\frac{1}{6}\right)^\frac{1}{p}r',$$
where we use the fact that over a region of length $1/2$ in $s$, $|\frac{\dd\alpha_E^j}{\dd s}| = 1$.

\medskip
\noindent\textbf{Error term $\td{w}_i - \td{L}_i$.}
From \eqref{est:wL}, $\|\frac{\dd\beta_E}{\dd s}(\td{w}_i-\td{L}_i)\|_{L^p(\C^n)} \leq C\epsilon^{2/p+\sigma} = o(r')$.

\medskip
\noindent\textbf{Error term $\constantVector_i$.}
Since $|\frac{\dd\beta_E}{\dd s}| \leq \frac{C}{h}$ on a support of length $2h$ in $s$, and $|\constantVector_i|_{\C^n} \leq r'$:
\[
\left\|\frac{\dd\beta_E}{\dd s}\constantVector_i\right\|_{L^p(\C^n)}
\leq \frac{C}{h}\cdot(2h)^{1/p}\cdot r'
= Ch^{-1/q}r'.
\]
By the definition of $h =  O(\ln\epsilon^{-1})$, this is $o(r')$.

\medskip
\noindent\textbf{Error term $\Pi_j\!\left(\frac{\dd\beta_E}{\dd s}\xi_i\right)$ — exponential decay.}
Note that $D_i\xi_i = g_i := (\Pi_i'-1)\eta_i$. By the exponential decay estimate for Cauchy--Riemann type operators \cite{robbin2001asymptotic}, since $g_i$ is supported in $[s_1-2h, s_1+2h]$:
\[
\|\xi_i\|_{L^p([s_0-2h,s_0]\times[0,1],\,M)}
\leq Ce^{-\lambda_1(s_1-s_0)}\|g_i\|_{L^p([s_1-2h,s_1+2h]\times[0,1],\,M)},
\]
where $\lambda_1 > 0$ is the spectral gap of the asymptotic operator of $D_i$.  Using $\|\cdot\|_{L^p(\C^n)} = \epsilon^{-1}\|\cdot\|_{L^p(M)}$, $e^{-\lambda_1(s_1-s_0)} = \epsilon^{2p\sigma\lambda_1/\delta}$, and $$\|g_i\|_{L^p(M)} \leq C \|\xi_i\|_{W^{1,p}(M)} \leq C\epsilon^{-2/p}\cdot \|\xi_i\|_{W^{1,p}_{(\epsilon)}(M)}$$
\begin{align}
\left\|\Pi_j\!\left(\frac{\dd\beta_E}{\dd s}\xi_i\right)\right\|_{L^p(\C^n)}
&\leq C\|\xi_i\|_{L^p([s_0-2h,s_0]\times[0,1],\,\C^n)} \notag\\
&= C\epsilon^{-1}\|\xi_i\|_{L^p([s_0-2h,s_0]\times[0,1],\,M)} \notag\\
&\leq C\epsilon^{-1}\cdot\epsilon^{\frac{2p\sigma\lambda_1}{\delta}}\cdot\epsilon^{-\frac{2}{p}}\cdot \|\xi_i\|_{W^{1,p}_{(\epsilon)}(M)} \notag\\
& = o(r'),
\end{align}
for $\delta$ small enough that $\frac{2p\sigma\lambda_1}{\delta} - \frac{2}{p} > 1 - \frac{1}{p}$.

\medskip
\noindent\textbf{Summary of the boundary 1 case:}
By the estimate above, $\|\Pi_j(\eta_V^\tau)\|_{L^p(\C^n)} \geq \frac{1}{2}\left(\frac{1}{6}\right)^\frac{1}{p}r' - o(r') > 0$ for $\epsilon$ sufficiently small, so $\Pi_j(\eta_V^\tau) \neq 0$ and $\obstructionSection_\tau \neq 0$.

\subsection*{Boundary 2: Over $B_V(r') \times \p(B_E(r'))$}

Note that 
\begin{align*}\p(B_E(r')) &= \left(\p(B_E^1(r')) \times B_E^2(r') \times B_E^3(r')\right) \\
&\quad\cup \left(B_E^1(r') \times \p(B_E^2(r')) \times B_E^3(r')\right) \\
&\quad\cup \left(B_E^1(r') \times B_E^2(r') \times \p(B_E^3(r'))\right),
\end{align*}
where
\[
B_E^i(r') = \left\{\constantVector_i = \sum_{j=1}^{n-\ind(\cP{p}_i)} \constantVector_i^j\,\alpha_E^i v_i^j \in \parameterSpace_E^i \,\bigg|\, \bigl(\sum_{j=1}^{n-\ind(\cP{p}_i)} |\constantVector_i^j|^p\bigr)^{1/p} = r'\right\}.
\]

For any $(\tdlv, \constantVector) \in B_V(r') \times \p(B_E(r'))$, we have 
\[
\left(\sum_{j=1}^{n-\ind(\cP{p}_i)} |\constantVector_i^j|^p\right)^{1/p} = r'
\]
from some $i \in \{1, 2, 3\}$. We estimate  $\Pi_i'(\eta_E^{i,\tau})$ in $L^p(M)$-norm.

From the homotopy definition,
\[
\eta_E^{i,\tau} = \left[D_i\constantVector_i + (1-\tau)\left(\frac{\dd\beta_{V}}{\dd s}(\tdlm+\td{\xi}-\td{L}_{i}+\tdlv) + F_i\right)\right]\otimes(\dd s + \iu\,\dd t).
\]
By the triangle inequality,
\begin{align}
\|\Pi_i'(\eta_E^{i,\tau})\|_{L^p(M)}
&\geq \|\Pi_i'(D_i\constantVector_i)\|_{L^p(M)} \notag\\
&\quad - C\left\|\frac{\dd\beta_V}{\dd s}(\tdlm - \td{L}_i)\right\|_{L^p(M)}
- C\left\|\frac{\dd\beta_V}{\dd s}\td{\xi}\right\|_{L^p(M)} \notag\\
&\quad - C\left\|\frac{\dd\beta_V}{\dd s}\tdlv\right\|_{L^p(M)}
- C\|F_i\|_{L^p(M)}.
\end{align}

\medskip
\noindent\textbf{Main term.}
\[
\|\Pi_i'(D_i\constantVector_i)\|_{L^p(M)}= \| D_i \constantVector_i \|_{L^p(M)} \geq \frac{1}{2} \left\|\sum \constantVector_i^j\frac{\dd\alpha_E^i}{\dd s}v_i^j\right\|_{L^p(M)} \geq \frac{1}{2}(1/6)^{1/p}\epsilon r',
\]
by Lemma~\ref{lemma: linearization}.

\medskip
\noindent\textbf{Error terms.}
From \eqref{est:L}: 
\[
\left\|\frac{\dd\beta_V}{\dd s}(\tdlm-\td{L}_i)\right\|_{L^p(M)}
= C \epsilon^{\frac{p}{\delta}\bigl(\pi(1+\sigma)-\frac{2}{p}\bigr)} h^{-\frac{1}{q}}
= o(\epsilon r'),
\]
if $\delta$ is small enough.

For the vertex-parameter term, we use the same cutoff estimate as above.  Since $\tdlv\in B_V(r')$, $\|\frac{\dd\beta_V}{\dd s}\|_{L^p}\leq Ch^{-1/q}$, and $\|\cdot\|_{L^p(M)}=\epsilon\|\cdot\|_{L^p(\C^n)}$,
\[
\left\|\frac{\dd\beta_V}{\dd s}\tdlv\right\|_{L^p(M)}
\leq C\epsilon h^{-1/q}r'
=o(\epsilon r').
\]

The nonlinear term is controlled by the estimates for the four terms in $F_i$ from the proof of Theorem~\ref{thm:contraction}.  Using \eqref{est:dbar}, \eqref{est:H-ball}, \eqref{est:P}, and \eqref{est:Q}, and then choosing $K$ large and $\epsilon$ small, we get
\[
\|F_i\|_{L^p(M)}\leq \frac{1}{100}\epsilon r'.
\]

From \eqref{est:xii2} (which gives $$\epsilon^{2/p-1}\|\frac{\dd\beta_V}{\dd s}\xi\|_{L^p(\C^n)} \leq C\epsilon^\sigma\|\xi\|_{W^{1,p}_\delta(\C^n)}$$ and $\|\cdot\|_{L^p(M)} = \epsilon\|\cdot\|_{L^p(\C^n)}$):
\begin{align}
\left\|\Pi_i'\!\left(\frac{\dd\beta_V}{\dd s}\td{\xi}\right)\right\|_{L^p(M)}
&\leq \epsilon\cdot C\epsilon^{1-\frac{2}{p}+\sigma}\|\xi\|_{W^{1,p}_\delta(\C^n)} \notag\\
&= C\epsilon^{2-\frac{2}{p}+\sigma}\|\xi\|_{W^{1,p}_\delta(\C^n)} \notag\\
&\leq C\epsilon^{2-\frac{2}{p}+\sigma}\cdot r  = C \epsilon^\sigma \epsilon r' =  o(\epsilon r').
\end{align}

\medskip
\noindent\textbf{Summary of the boundary 2 case:} 
Combining gives
\[
\|\Pi_i'(\eta_E^{i,\tau})\|_{L^p(M)}
\geq \left(\frac{1}{2}(\frac{1}{6})^{1/p}-\frac{1}{100}\right)\epsilon r' - o(\epsilon r')
 > 0
\]
for $\epsilon$ sufficiently small, so $\Pi_i'(\eta_E^{i,\tau}) \neq 0$ and $\obstructionSection_\tau \neq 0$.

\end{proof}
Now Theorem~\ref{thm: main} follows from Theorem~\ref{thm: degree one} and Proposition~\ref{prop: homotopy}.

\appendix

\section{Sobolev Constant and Quadratic Estimate on Thin Domains}
\newcommand{\vol}{\operatorname{vol}}
\subsection{Weighted Sobolev constants}

Weighted Sobolev norms play a central role in gluing analysis for
$J$-holomorphic curves.  McDuff--Salamon~\cite{mcduff2012j} use a
conformal weight to give comparable volume to the neck and the bubble in
sphere connected-sum gluing.  Fukaya--Oh~\cite{fukaya1997zeroloop} use an
adiabatic weight adapted to open strings in the cotangent bundle.  We now
compare these two settings.

Let $(\Sigma,j)$ be a Riemann surface, possibly noncompact and possibly
with boundary or punctures.  Following \cite[Section~10.3]{mcduff2012j},
weighted norms can be defined by conformally rescaling the coordinate metric
$ds^2+dt^2$.  Let $\theta(s,t)>0$ and put
\[
  g=\theta^{-2}(ds^2+dt^2),
  \qquad
  \mu=\vol^g=\theta^{-2}\,ds\wedge dt.
\]
For $p>2$, the corresponding norms for a section $\xi$ and a one-form $\eta$
are
\begin{align}
  \|\xi\|_{W^{1,p}(\Sigma,\mu)}
  &:=\left(\int_\Sigma
    \bigl(\theta^{-2}|\xi|^p
    + \theta^{p-2}|\nabla\xi|^p\bigr)
    ds\,dt\right)^{1/p},
  \label{eq:weighted-W1p}\\[4pt]
  \|\eta\|_{L^p(\Sigma,\mu)}
  &:=\left(\int_\Sigma
    \theta^{p-2}|\eta|^p\,ds\,dt\right)^{1/p}.
  \label{eq:weighted-Lp}
\end{align}
Here the pointwise norms on the right-hand side are computed using the
unrescaled coordinate metric.  The factor $\theta^{p-2}$ in the one-form
norm is the usual conformal scaling: a one-form has pointwise norm multiplied
by $\theta$, while the volume form is multiplied by $\theta^{-2}$.

The associated $W^{1,p}$ Sobolev constant is
\begin{equation}\label{eq:sobolev-constant}
  C_p(\Sigma,\mu) := \sup_{0 \neq f \in C^\infty \cap W^{1,p}(\Sigma,\mu)}
  \frac{\|f\|_{L^\infty(\Sigma)}}{\|f\|_{W^{1,p}(\Sigma,\mu)}}.
\end{equation}
After increasing this constant by a factor depending only on the rank of
$E$ and a fixed choice of local trivializations, the same estimate holds for
sections of a Euclidean vector bundle $E\to\Sigma$.

We next relate the weighted norm on the rescaled strip to the standard norm
on the physical thin strip.  Let
\[
  Z_{L,\eps}=[0,L/\eps]\times[0,1],
  \qquad
  \Omega_{L,\eps}=[0,L]\times[0,\eps],
\]
and let
\[
  \rho_\eps:Z_{L,\eps}\to\Omega_{L,\eps},
  \qquad
  \rho_\eps(s,t)=(\eps s,\eps t).
\]
If $u$ is a section over $\Omega_{L,\eps}$ and
$\xi=u\circ\rho_\eps$, then
\begin{equation}\label{eq:FO-W1p}
  \|\xi\|_{W^{1,p}_{(\eps)}(Z_{L,\eps})}
  =
  \|u\|_{W^{1,p}(\Omega_{L,\eps})},
\end{equation}
where
\[
  \|\xi\|_{W^{1,p}_{(\eps)}(Z_{L,\eps})}
  :=
  \left(\int_{Z_{L,\eps}}
  \eps^2|\xi|^p+\eps^{2-p}|\nabla\xi|^p\,ds\,dt
  \right)^{1/p}.
\]
Similarly, if $\gamma$ is a one-form over $\Omega_{L,\eps}$ and
$\eta=\rho_\eps^*\gamma$, then
\begin{equation}\label{eq:FO-Lp}
  \|\eta\|_{L^p_{(\eps)}(Z_{L,\eps})}
  =
  \|\gamma\|_{L^p(\Omega_{L,\eps})},
\end{equation}
where
\[
  \|\eta\|_{L^p_{(\eps)}(Z_{L,\eps})}
  :=
  \left(\int_{Z_{L,\eps}}
  \eps^{2-p}|\eta|^p\,ds\,dt
  \right)^{1/p}.
\]
Thus Fukaya--Oh's weighted norms on the rescaled strip are exactly the
standard norms on the physical thin strip.

As in \cite[Chapter~10]{mcduff2012j}, the Sobolev constant
$C_p(\Sigma,\mu)$ enters the quadratic estimates in the gluing argument and
controls the size of the neighborhood on which the implicit function theorem
applies.  In the McDuff--Salamon setting the conformal weight $\theta^R$ is
chosen so that $C_p(\Sigma,\mu)$ remains bounded uniformly in the gluing
parameter $R\to\infty$; see \cite[Lemma~10.3.1]{mcduff2012j}.  In the
Fukaya--Oh adiabatic setting, the constant weight
$\theta^\eps\equiv1/\eps$ makes the relevant Sobolev constants blow up as
$\eps\to0$.  This blowup reflects the collapse of the thin strip.  We
estimate it below.

\subsection{Quadratic estimates}

Let $(X,\omega,J)$ be a compact almost Hermitian manifold, and let
$u\colon\Sigma\to X$ be a smooth map with $\|\dd u\|_{L^p}<\infty$.  Put
$E=u^*TX$.  In the application $X=T^*M$ is noncompact, but all maps lie in a
fixed compact neighborhood of the zero section for $\eps$ sufficiently small.
The same estimates then hold with constants depending on that compact set.

For a section $\xi$ of $E$ sufficiently small, define
\[
\mathcal F_u(\xi)
:=
\operatorname{Pal}_{\xi}^{-1}\bigl(\dbar_J(\exp_u\xi)\bigr),
\qquad
D_u:=d\mathcal F_u(0),
\]
where $\operatorname{Pal}_{\xi}$ denotes parallel transport along the
geodesic $r\mapsto \exp_u(r\xi)$.  The following estimate keeps track of the
Sobolev constant.

\begin{proposition}[Quadratic estimate]\label{prop:app-quadratic-estimate}
There are constants $K>0$ and $h_0>0$, depending only on $(X,\omega,J)$ and
$\|\dd u\|_{L^p}$, such that
\begin{equation}\label{eq:app-quadratic-estimate}
\bigl\|d\mathcal F_u(\xi)\xi'-D_u\xi'\bigr\|_{L^p(\Sigma,\mu)}
\le
K C_p^2(\Sigma,\mu)
\|\xi\|_{W^{1,p}(\Sigma,\mu)}
\|\xi'\|_{W^{1,p}(\Sigma,\mu)}
\end{equation}
whenever $\|\xi\|_{L^{\infty}(\Sigma)}\le h_0$.
\end{proposition}

\begin{proof}
The standard nonlinear estimate for the Cauchy--Riemann operator in
exponential coordinates gives the pointwise bound
\[
\bigl|d\mathcal F_u(\xi)\xi'-D_u\xi'\bigr|
\le
C\bigl(|\dd u|\,|\xi|\,|\xi'|+|\xi'|\,|\nabla\xi|+|\xi|\,|\nabla\xi'|\bigr),
\]
where $C$ depends only on $(X,\omega,J)$; see
\cite[Proposition~3.5.3]{mcduff2012j}.  Taking the $L^p(\Sigma,\mu)$-norm,
estimating the undifferentiated factors in $L^\infty$, and then using
\eqref{eq:sobolev-constant} gives \eqref{eq:app-quadratic-estimate}, after
absorbing the bundle-trivialization constants and $1+\|\dd u\|_{L^p}$ into
$K$.
\end{proof}
\begin{remark}
    When $\|\dd u\|_{L^p}$ is small and $C_p(\Sigma,\mu)$ is big, a sharper inequality from the above proof is 
    \begin{multline}       
    \label{eq:app-quadratic-estimate}
\bigl\|d\mathcal F_u(\xi)\xi'-D_u\xi'\bigr\|_{L^p(\Sigma,\mu)} \\  
\le C C_p(\Sigma,\mu)\Big(1+C_p(\Sigma,\mu)\|\dd u\|_{L^p(\Sigma,\mu)}\Big)
\|\xi\|_{W^{1,p}(\Sigma,\mu)}
\|\xi'\|_{W^{1,p}(\Sigma,\mu)}.
\end{multline}    
\end{remark}

The following consequence is the form used in the contraction mapping
argument.  It is the usual Floer--Picard estimate; for the attribution, see
\cite[Proposition~24, p.~25]{FloerMonopoles} and \cite[Lemma~A.5]{Lipshitz06}.

\begin{lemma}[Floer--Picard estimate]\label{lem:app-floer-picard}
Take the same constants $K$ and $h_0$ as in Proposition~\ref{prop:app-quadratic-estimate}.  Let
\[
N(\xi):=\mathcal F_u(\xi)-\mathcal F_u(0)-D_u\xi.
\]
Suppose that $\xi_0$ and $\xi_1$ are sections such that
\[
  \|(1-\tau)\xi_0+\tau\xi_1\|_{L^\infty(\Sigma)}\le h_0
  \qquad\text{for all }\tau\in[0,1].
\]
Then
\begin{equation}\label{eq:app-floer-picard}
\begin{aligned}
\|N(\xi_1)-N(\xi_0)\|_{L^p(\Sigma,\mu)}
&\le
K C_p^2(\Sigma,\mu)
\bigl(\|\xi_0\|_{W^{1,p}(\Sigma,\mu)}+
      \|\xi_1\|_{W^{1,p}(\Sigma,\mu)}\bigr)\\
&\hspace{3cm}\cdot
\|\xi_1-\xi_0\|_{W^{1,p}(\Sigma,\mu)}.
 \end{aligned}
\end{equation}
In particular, the $L^\infty$ smallness condition holds if
\[
  \|\xi_0\|_{W^{1,p}(\Sigma,\mu)}+
  \|\xi_1\|_{W^{1,p}(\Sigma,\mu)}
  \le h_0/C_p(\Sigma,\mu).
\]
\end{lemma}

\begin{proof}
By the fundamental theorem of calculus in Banach spaces,
\[
N(\xi_1)-N(\xi_0)
=
\int_0^1
\bigl(d\mathcal F_u(\xi_0+\tau(\xi_1-\xi_0))-D_u\bigr)
(\xi_1-\xi_0)\,\dd \tau.
\]
Apply Proposition~\ref{prop:app-quadratic-estimate} to
$\xi_0+\tau(\xi_1-\xi_0)$ and $\xi_1-\xi_0$, and use
\[
\|\xi_0+\tau(\xi_1-\xi_0)\|_{W^{1,p}(\Sigma,\mu)}
\le
\|\xi_0\|_{W^{1,p}(\Sigma,\mu)}+
\|\xi_1\|_{W^{1,p}(\Sigma,\mu)}.
\]
This proves \eqref{eq:app-floer-picard}.  The final sentence follows from
the definition of $C_p(\Sigma,\mu)$.
\end{proof}

\subsection{Thin-domain Sobolev constants}

The standard input is Morrey's inequality on fixed two-dimensional domains.
In the gluing argument we also need the following anchored version on a long
strip with the $\eps$-weighted norm.  This records the dependence on the
strip length explicitly and is the estimate used in the proof of
Theorem~\ref{thm:contraction}.

\begin{proposition}[Anchored Sobolev estimate on a thin ribbon]\label{prop:app-anchored-thin-ribbon}
Let $p>2$ and let $q$ be the conjugate exponent, $1/p+1/q=1$.  For $S\geq1$
and $0<\epsilon\leq1$, put
\[
R_S=[0,S]\times[0,1]
\]
and define
\[
\|a\|_{W^{1,p}_{(\epsilon)}(R_S)}
:=
\left(
\int_{R_S}\epsilon^2 |a|^p+\epsilon^{2-p}|\nabla a|^p\,\dd s\,\dd t
\right)^{1/p}.
\]
For each rank $m\geq1$ there is a constant $C=C(p,m)$, independent of $S$
and $\epsilon$, such that every $a\in W^{1,p}(R_S;\R^m)$ satisfying
\[
\int_0^1 a(0,t)\,\dd t=0
\]
obeys
\begin{equation}\label{eq:app-anchored-thin-ribbon}
\|a\|_{L^\infty(R_S)}
\leq
C(S+1)^{1/q}\epsilon^{1-\frac2p}
\|a\|_{W^{1,p}_{(\epsilon)}(R_S)}.
\end{equation}
\end{proposition}

\begin{proof}
It is enough to prove the scalar case; the vector-valued case follows
componentwise after increasing the constant.  Since $p>2$, we use the
continuous representative of $a$.  The condition
$\int_0^1 a(0,t)\,\dd t=0$ implies that there is $t_0\in[0,1]$ with
$a(0,t_0)=0$.

Let $N=\lceil S\rceil$.  Extend $a$ from $R_S$ to
$R_N=[0,N]\times[0,1]$ by reflection across the boundary $s=S$ if necessary.
This changes the $L^p$-norms of $a$ and $\nabla a$ by at most a universal
factor, which we absorb into the constant.  Thus we may assume $S=N$ is an
integer.  Cover $R_N$ by the unit squares
\[
Q_j=[j,j+1]\times[0,1],\qquad j=0,\dots,N-1.
\]
On the unit square $Q=[0,1]^2$, the usual Sobolev--Morrey embedding, see
\cite{AF}, together with the Poincar\'e inequality gives
\begin{equation}\label{eq:app-unit-oscillation}
\operatorname{osc}_Q v
\leq C_p\|\nabla v\|_{L^p(Q)}
\end{equation}
for every $v\in W^{1,p}(Q)$.  Indeed, subtract the average $v_Q$ and apply
the fixed-domain embedding $W^{1,p}(Q)\hookrightarrow C^0(Q)$ to $v-v_Q$.

Fix $x\in R_N$.  If $x\in Q_j$, connect $(0,t_0)$ to $x$ through the chain of
adjacent squares $Q_0,\dots,Q_j$; for instance, use the intermediate boundary
points $(k,1/2)$, $k=1,\dots,j$.  Applying
\eqref{eq:app-unit-oscillation} on each square in the chain gives
\[
|a(x)|
\leq
C_p\sum_{k=0}^{j}\|\nabla a\|_{L^p(Q_k)}.
\]
Taking the supremum over $x$ and using H\"older's inequality for the finite
sum,
\[
\|a\|_{L^\infty(R_N)}
\leq
C_p N^{1/q}\left(\sum_{k=0}^{N-1}\|\nabla a\|_{L^p(Q_k)}^p\right)^{1/p}
\leq
C_p(S+1)^{1/q}\|\nabla a\|_{L^p(R_S)}.
\]
Finally, the weighted norm gives
\[
\|\nabla a\|_{L^p(R_S)}
\leq
\epsilon^{1-2/p}\|a\|_{W^{1,p}_{(\epsilon)}(R_S)}.
\]
This proves \eqref{eq:app-anchored-thin-ribbon}.
\end{proof}

The following elementary estimate records the Sobolev--Morrey constant on a
physical thin product domain.  The proof uses even reflections.

\begin{proposition}[Sobolev--Morrey constant on a thin ribbon]\label{prop:app-thin-morrey}
Let $\Omega_\eps = [0,L]\times[0,\eps]$ with $0 < \eps\leq 1$ and
$\eps\leq L$, and let $p>2$ and $\alpha=1-2/p$.  There exists
$C_p>0$, independent of both $\eps$ and $L$, such that for every
$u\in W^{1,p}(\Omega_\eps)$,
\begin{equation}\label{eq:app-thin-morrey}
  \norm{u}_{C^{0,\alpha}(\Omega_\eps)}
  \leq C_p\,\max\left\{1,\left(\frac{2}{L}\right)^{1/p}\right\}
  \eps^{-1/p}\norm{u}_{W^{1,p}(\Omega_\eps)}.
\end{equation}
\end{proposition}

\begin{proof}
Set $N_t := \lfloor 1/\eps \rfloor \geq 1$, so
$h := N_t\eps \in [1/2,1]$.  Reflecting $u$ evenly across the successive
boundaries $t=k\eps$ ($k=0,\ldots,N_t-1$) produces
$E_tu\in W^{1,p}([0,L]\times[0,h])$.  Since each of the $N_t$ copies is
isometric to $\Omega_\eps$,
\begin{equation}\label{eq:norm-t}
  \|E_tu\|_{W^{1,p}([0,L]\times[0,h])}^p
  = N_t\|u\|_{W^{1,p}(\Omega_\eps)}^p
  \leq \eps^{-1}\|u\|_{W^{1,p}(\Omega_\eps)}^p.
\end{equation}

Set $N_s := \lfloor 1/L \rfloor + 1 \geq 1$, so
$\ell := N_sL \in [1,2)$.

\emph{Case $L < 1$.}
Reflecting $E_tu$ evenly across the boundaries $s=jL$
($j=0,\ldots,N_s-1$) produces
$v:=E_sE_tu\in W^{1,p}([0,\ell]\times[0,h])$.  Since each of the $N_s$
copies is isometric to $[0,L]\times[0,h]$,
\begin{equation}\label{eq:norm-s}
  \|v\|_{W^{1,p}([0,\ell]\times[0,h])}^p
  = N_s\|E_tu\|_{W^{1,p}([0,L]\times[0,h])}^p
  \leq \frac{2}{L}\|E_tu\|_{W^{1,p}([0,L]\times[0,h])}^p.
\end{equation}
Set $\widehat\Omega := [0,\ell]\times[0,h]$.  Combining
\eqref{eq:norm-t} and \eqref{eq:norm-s} gives
\begin{equation}\label{eq:norm-combined}
  \|v\|_{W^{1,p}(\widehat\Omega)}^p
  \leq \frac{2}{L\eps}\|u\|_{W^{1,p}(\Omega_\eps)}^p.
\end{equation}
By construction $\ell\in[1,2)$ and $h\in[1/2,1]$, so
$\widehat\Omega$ has both dimensions bounded above and below by universal
constants.  The Sobolev--Morrey embedding on $\widehat\Omega$ gives
\begin{equation}\label{eq:app-fixed-domain-sm}
  \|v\|_{C^{0,\alpha}(\widehat\Omega)}
  \leq C_p\|v\|_{W^{1,p}(\widehat\Omega)},
\end{equation}
where $C_p$ depends only on $p$.  Since $v|_{\Omega_\eps}=u$, restricting and
combining \eqref{eq:norm-combined}--\eqref{eq:app-fixed-domain-sm} yields
\eqref{eq:app-thin-morrey} for $L<1$.

\emph{Case $L \geq 1$.}
We have $N_s=1$ and $\ell=L$, so set $v:=E_tu$ and
$\widehat\Omega:=[0,L]\times[0,h]$.  Then \eqref{eq:norm-t} gives
\[
\|v\|_{W^{1,p}(\widehat\Omega)}^p
\leq \eps^{-1}\|u\|_{W^{1,p}(\Omega_\eps)}^p.
\]
Cover $\widehat\Omega$ by rectangles $Q_j$, $j=1,\ldots,M$, with bounded
overlap.  Each $Q_j$ has horizontal length at most $1$ and vertical height
$h\in[1/2,1]$, so the fixed-domain Sobolev--Morrey estimate applies on each
$Q_j$ with the same constant $C_p$.  Taking the supremum over $j$ gives the
$C^0$ bound.  For the H\"older seminorm, use the local estimate when
$|x-y|\le1$ and the $C^0$ bound when $|x-y|>1$.  Hence
\[
\|v\|_{C^{0,\alpha}(\widehat\Omega)}
\leq C_p\|v\|_{W^{1,p}(\widehat\Omega)},
\]
with $C_p$ independent of $L$.  Restricting to $\Omega_\eps$ yields
\eqref{eq:app-thin-morrey} since
$\max\{1,(2/L)^{1/p}\}=1$ for $L\geq2$ and is bounded by $2^{1/p}$ for
$1\leq L<2$.
\end{proof}

The previous estimate gives the half-strip estimate used in the main text.

\begin{proposition}[Weighted Sobolev estimate on the half-strip]\label{prop:app-weighted-half-strip-sobolev}
Let $Z=[0,\infty)\times[0,1]$, let $p>2$, and let $0<\epsilon\le1$.  Define
\[
\|u\|_{W^{1,p}_{(\epsilon)}(Z)}
:=
\left(
\int_Z \epsilon^2|u|^p+\epsilon^{2-p}|\nabla u|^p\,\dd s\,\dd t
\right)^{1/p}.
\]
Then there is a constant $C=C(p)$, independent of $\epsilon$, such that
\begin{equation}\label{eq:app-weighted-half-strip-sobolev}
\|u\|_{L^\infty(Z)}
\le
C\epsilon^{-1/p}\|u\|_{W^{1,p}_{(\epsilon)}(Z)}.
\end{equation}
The same estimate holds for vector-valued sections after increasing $C$ by a
factor depending only on the rank and on fixed local trivializations.
\end{proposition}

\begin{proof}
It is enough to prove the scalar case.  Fix $(s_*,t_*)\in Z$ and put
$R=\epsilon^{-1}$.  Choose $a\ge0$ such that
$s_*\in[a,a+R]$, and set
\[
  Q=[a,a+R]\times[0,1].
\]
Let $\Omega=[0,1]\times[0,\epsilon]$ and define
\[
  v(x,y)=u(a+x/\epsilon,y/\epsilon).
\]
Then the scaling identity gives
\[
  \|v\|_{W^{1,p}(\Omega)}=
  \|u\|_{W^{1,p}_{(\epsilon)}(Q)}.
\]
By Proposition~\ref{prop:app-thin-morrey}, applied with $L=1$,
\[
|u(s_*,t_*)|
\le
\|v\|_{L^\infty(\Omega)}
\le
C\epsilon^{-1/p}\|v\|_{W^{1,p}(\Omega)}
\le
C\epsilon^{-1/p}\|u\|_{W^{1,p}_{(\epsilon)}(Z)}.
\]
Taking the supremum over $(s_*,t_*)$ proves the estimate.
\end{proof}

The following lemma is used to prove exponential decay for $\xi_i$ and $\xi$
in the homotopy argument for the obstruction section.  The result is
standard, but we include the proof because we need this precise form.

\begin{lemma}[Exponential localization]\label{lem:exp-loc}
Let $Y$ be compact, and consider the product $[0,\infty)\times Y$.  Suppose
\[
  D = \partial_s + L(s)\colon W^{1,p}([0,\infty)\times Y)
  \to L^p([0,\infty)\times Y)
\]
is an isomorphism between Banach spaces, where $L(s)$ is a first-order
Fredholm differential operator on $Y$.  Suppose $D\xi=\eta$ and
$\operatorname{supp}\eta\subset[0,R_0]\times Y$.  Then for every
$0<d<\|D^{-1}\|^{-1}$, every $h>0$, and all $s\ge R_0$,
\begin{equation}\label{eq:exp-decay}
  \|\xi\|_{W^{1,p}([s,s+h]\times Y)}
  \leq
  C_d\, e^{-d(s-R_0)}\|\xi\|_{W^{1,p}([0,\infty)\times Y)},
\end{equation}
where $C_d>0$ depends on the operator norms of $D$ and $D^{-1}$ and on $d$.
\end{lemma}

\begin{proof}
For $0<d<\|D^{-1}\|^{-1}$, the conjugated operator
$e^{ds}De^{-ds}=D-dI$ is invertible by a Neumann series argument.  Thus $D$
remains an isomorphism on the exponentially weighted spaces
$W^{1,p}_d\to L^p_d$, with
\begin{equation}\label{eq:weighted-bound}
  \|\xi\|_{W^{1,p}_d} \leq C_d\|\eta\|_{L^p_d}.
\end{equation}
Since $\operatorname{supp}\eta\subset[0,R_0]\times Y$, we have
$\|\eta\|_{L^p_d}\leq e^{dR_0}\|\eta\|_{L^p}$.  For $s'\ge s\ge R_0$, the
weight satisfies $e^{ds'}\ge e^{ds}$, so
\[
\|\xi\|_{W^{1,p}([s,s+h]\times Y)}
\leq e^{-ds}\|\xi\|_{W^{1,p}_d}.
\]
Combining these estimates with
$\|\eta\|_{L^p}\leq\|D\|\,\|\xi\|_{W^{1,p}([0,\infty)\times Y)}$ gives
\eqref{eq:exp-decay}, after absorbing $\|D\|$ into $C_d$.
\end{proof}

\begin{remark}
In our applications, $Y=[0,1]$, $L(s)=J_0\frac{d}{dt}$ for the local model,
and
\[
L(s)=J(\bar\chi_i^\varepsilon(s))
\bigl(\nabla_t+\nabla X_{\varepsilon(f_{i+1}-f_i)}\bigr)
\]
for the edge model.  In both cases $D=\partial_s+L(s)$ satisfies the
hypotheses of Lemma~\ref{lem:exp-loc}.
\end{remark}

\bibliographystyle{alpha}
\bibliography{mybib}

\end{document}